\documentclass{amsproc}
\pdfoutput=1
\usepackage{amsmath}
\usepackage{amssymb}
\usepackage{rotating}
\usepackage{graphicx}
\DeclareGraphicsExtensions{.png,.pdf,.eps}
\usepackage[all]{xy}
\usepackage{tikz}
\usepackage{enumitem}
\usepackage{multirow,tabularx}
\usepackage[pdfstartview=FitH]{hyperref}

\newtheorem{theorem}[equation]{Theorem}
\newtheorem{lemma}[equation]{Lemma}
\newtheorem{corollary}[equation]{Corollary}
\newtheorem{proposition}[equation]{Proposition}

\theoremstyle{definition}

\newtheorem{definition}[equation]{Definition}

\newtheorem{remark}[equation]{Remark}

\newtheorem{example}[equation]{Example}

\numberwithin{equation}{section}

\def\bC{{\mathbb C}}

\def\bQ{{\mathbb Q}}
\def\bR{{\mathbb R}}
\def\bZ{{\mathbb Z}}
\def\bT{{\mathbb T}}

\def\cO{{\mathcal{O}}}
\def\cR{{\mathcal R}}

\def\eps{\varepsilon}

\def\ra{\rightarrow}

\def\wt{\widetilde}
\def\ovl{\overline}

\def\Ab{\operatorname{Ab}}

\def\Hilb{\operatorname{Hilb}}
\def\Hom{\operatorname{Hom}}

\def\Spec{\operatorname{Spec}}

\def\age{\operatorname{age}}

\def\diag{\operatorname{diag}}
\def\div{\operatorname{div}}

\def\deg{\operatorname{deg}}

\def\det{\operatorname{det}}

\def\Mov{{\operatorname{Mov}}}

\def\Cl{\operatorname{Cl}}

\def\SL{\operatorname{SL}}
\def\GL{\operatorname{GL}}
\def\SO{\operatorname{SO}}
\def\BD{\operatorname{BD}}
\def\BT{\operatorname{BT}}
\def\BO{\operatorname{BO}}
\def\BI{\operatorname{BI}}

\definecolor{zielony}{rgb}{0.5, 0.9, 0.1}
\definecolor{czerwony}{rgb}{0.9, 0.2, 0.1}
\definecolor{niebieski}{rgb}{0.3, 0.1, 0.9}

\begin{document}

\title[Crepant resolutions of 3-dimensional quotients]
 {Crepant resolutions of 3-dimensional quotient singularities via Cox rings}

\author[M.~Donten-Bury]{Maria Donten-Bury}
\address{Instytut Matematyki UW, Banacha 2, 02-097 Warszawa, Poland}
\email{M.Donten@mimuw.edu.pl, M.Grab@mimuw.edu.pl}

\author[M.~Grab]{Maksymilian Grab}



\date{\today}


\begin{abstract}
We study Cox rings of crepant resolutions of quotient singularities $\bC^3/G$ where $G$ is a finite subgroup of $\SL(3,\bC)$. We use them to obtain information on the geometric structure of these resolutions, number of different resolutions and relations between them. In particular, we treat explicitly several examples where $G$ contains elements of age~2.
\end{abstract}

\maketitle

\section{Introduction}

It is long known that for du Val singularities, i.e. quotients $\bC^2/G$ for finite subgroups $G \subset \SL(2,\bC)$, there is a unique minimal resolution, which is crepant. The study of these resolutions led McKay to observing a one-to-one correspondence between the components of the exceptional locus and irreducible representations of $G$, or its conjugacy classes. Currently many variants of the \emph{McKay correspondence} for higher dimensional quotient singularities are considered; some of them are already theorems (see e.g.~\cite{ItoReid} for a case most relevant to this article), some still conjectures.

In this paper we concentrate on resolutions of 3-dimensional quotient singularities, which have also been studied extensively. It is known, see~\cite{ItoRes3, ItoMonomial, Markushevich168, RoanRes3}, that a crepant resolution of a quotient singularity $\bC^3/G$, where $G$ is a subgroup of $\SL(3,\bC)$, exists. Such a resolution is usually non-unique, but all crepant resolutions of a quotient singularity differ by a small $\bQ$-factorial modification, a sequence of flops. Thus it is natural to investigate the whole (finite) family of crepant resolutions, not only a chosen example.

As shown in~\cite{BKR}, a crepant resolution of $\bC^3/G$ can be constructed as an invariant Hilbert scheme $G$-$\Hilb$, which was conjectured by Nakamura in~\cite{GHilb} (see also~\cite{CrawReid} for the case of $G$ abelian). In~\cite{CrawIshii} all small $\bQ$-factorial modifications of $G$-$\Hilb$ for $G$ abelian are analysed and constructed as certain moduli spaces. However, in general this point of view does not seem sufficient for finding the set of all crepant resolutions of a singularity, or even for determining their number. Another successful approach to investigating such resolutions (which can be considered an extension of $G$-Hilbert schemes) goes via quiver representation theory, see e.g.~\cite{superpotentials, homologicalMMP, NdCS}. However, up to now significant results have been obtained only for groups not containing any elements of age~2. Geometrically, this condition is equivalent to saying that the fibre over the origin consists only of 1-dimensional components (see Theorem~\ref{theorem_mckay} and Corollary~\ref{corollary_mckay}, or the original source~\cite{ItoReid}), which apparently makes these resolutions easier to deal with. See e.g.~\cite{NdCS} for application of this method to finite subgroups $G \subset \SO(3)$, in particular to representations of dihedral groups.

Our aim is to describe Cox rings of crepant resolutions of 3-dimensional quotient singularities and use them to get more insight into the structure of the set of all crepant resolutions of such a singularity. Our study of Cox rings of resolutions of quotient singularities, initiated in~\cite{81resolutions} and developed in~\cite{SymplCox, CompCox} in two different directions (symplectic and algorithmic), gives us a solid basis for the work in the 3-dimensional case. We start from producing a finite generating set of the Cox ring $\cR(X_0)$ of a crepant resolution $X_0 \ra \bC^3/G$, based on the structure of the Cox ring $\cR(\bC^3/G)$ of the singularity and using the McKay correspondence to take the information on the exceptional divisors into account. See Proposition~\ref{prop_structure_cox_ring} for the details; note that although it characterises all elements of $\cR(X_0)$, finding a small generating set is still a nontrivial combinatorial problem. We solve it in a number of cases (listed below), thus obtaining $\cR(X_0)$ in a form suitable for further analysis and computations. In particular, we are able to describe in detail several cases where $G$ has elements of age~2, i.e. there are divisorial components in the fibre over the origin.

 The Cox ring $\cR(X_0)$ contains in a natural way the information on all resolutions, which differ from $X_0$ by a sequence of flops. That is, if one knows the Cox ring $\cR(X_0)$, one can extract from it the subdivision of the movable cone $\Mov(X_0)$ into Mori (or GIT) chambers, using methods from toric geometry. In particular, one can find the number of all crepant resolutions of a given quotient singularity and describe how different resolutions are linked by sequences of flops. Moreover, the geometric structure of the resolution, i.e. components of the exceptional set and their intersections, can be determined from the structure of~$\cR$ via Geometric Invariant Theory (this, however, usually requires a lot of work). For example, one can observe explicitly the behaviour of flops between resolutions.

 We realise this program in three different settings. First, in section~\ref{section_dihedral} we consider~$G$ being a representation of a dihedral group, hence $G \subset \SO(3)$ and it does not contain elements of age~2. We confirm the results of~\cite{NdCS}, give a description of the chamber structure of the movable cone and, in the odd case, of the geometry of resolutions via GIT. This section should be treated as an extended example of our methods. Next, in section~\ref{section_reducible} we study non-abelian $G \subset \SL(3,\bC)$ which are reducible representations. That is, they can be decomposed as a product of a 2-dimensional and a 1-dimensional representation of the abstract group isomorphic to $G$, not necessarily faithful. Note that we distinguish the reducible case from the irreducible one, because in the former the Cox ring of a crepant resolution is expected to have a simpler structure: only one trinomial relation between its generators. We prove this statement and provide the single trinomial relation for the Cox ring. We also present two examples of~$G$ which contain elements of age~2 (as most of such groups) and give some information on their Mori chamber structure in the cone of movable divisors. Finally, in section~\ref{section_irreducible} we investigate three cases where~$G$ is an irreducible representation: a 21-element (trihedral) group which is the smallest example with a divisor in the fibre over the origin, a 27-element (trihedral) group which is a representation of the Heisenberg group, and a 54-element (non-trihedral) group which is a double extension of the previous one. We give generators for their Cox rings, compute the number of Mori chambers, and in the Appendix we provide some data concerning geometry of the resolutions in the 21-element case. Finding the general formula for the Cox ring generators seems to be much harder in the irreducible case, but already studying presented examples gives new information on crepant resolutions of 3-dimensional quotient singularities.

\subsection*{Acknowledgements}

The authors are very grateful to Michael Wemyss for asking questions which inspired this work. They would like to thank Jaros\l{}aw Wi\'sniewski for discussions and Simon Keicher for his help with some useful software packages.

The first author was supported by the Polish National Science Center project 2013/08/A/ST1/00804. This project was started when she held a Dahlem Research School Postdoctoral Fellowship at Freie Universit\"at Berlin. The second author was supported by the Polish National Science Center project 2015/17/N/ST1/02329.

\section{Description of the methods}

\subsection{Cox rings of resolutions of quotient singularities}\label{section_methods_Cox}

Let $X$ be a normal algebraic variety with finitely generated class group $\Cl(X)$ and only constant invertible regular functions. The Cox ring of~$X$ is
\begin{equation*}
\cR(X) = \bigoplus_{[D]\in \Cl(X)}H^{0}(X,\cO_X(D)),
\end{equation*}
with the multiplication of sections coming from identifying them with rational functions: $H^{0}(X,\cO_X(D)) = \{f\in \bC(X)^{*}\ : \ \div f + D \ge 0\}\cup \{0\} \subset \bC(X)$. For the details check~\cite[Sect.~1.4]{CoxRings}.

Assume that $\cR(X)$ is finitely generated as a $\bC$-algebra. Then we may consider a variety $\Spec \cR(X)$, endowed with a natural action of the Picard torus $\bT_X=\Hom(\Cl(X),\bC^*)$, coming from the $\Cl(X)$-grading on $\cR$. In this case a lot of information about birational geometry of $X$ can be read out of $\Spec \cR (X)$ with the $\bT_X$-action via Geometric Invariant Theory: one can reconstruct not only $X$, but also its small $\bQ$-factorial modifications, as GIT quotients of $\Spec \cR$ by $\bT_X$ (under some technical assumptions which will be satisfied in our setting), see~\cite[Sect.~3.3]{CoxRings}. The set of small $\bQ$-factorial modifications of~$X$ is in one-to-one correspondence with the set of GIT chambers in the cone of movable divisors $\Mov(X)$; we continue this topic in section~\ref{section_methods_Mov}.

We recall what is known about Cox rings of crepant resolutions of quotient singularities $\bC^{n}/G$ for finite subgroups $G\subset \SL(n,\bC)$ in general, though in this article we are interested only in the case $n=3$. The idea is to understand the Cox ring $\cR(X)$ of a crepant resolution $X$ without knowing this resolution before, and then to use $\cR(X)$ to recover geometric data, such as the structure of $X$ and other resolutions, which are its small $\bQ$-factorial modifications, and also relations between them. The main tool used to determine the Cox ring is the McKay correspondence (as in~\cite{ItoReid}), which allows us to read out some information about the geometry of crepant resolutions in codimension~1 from the structure of the group~$G$. 

\begin{definition}
A~\textbf{resolution of singularities} is a proper birational morphism $\varphi:X\to Y$ which is an isomorphism outside the singular point set of~$Y$. A~resolution is called \textbf{projective} if $\varphi$ is a projective morphism. It~is \textbf{crepant} if $\varphi^*K_Y = K_X$.
\end{definition}

We take $Y = \bC^n/G$ for a finite subgroup $G \subset \SL(n,\bC)$. All resolutions considered in this article will be projective. Note that a crepant resolution, if it exists, is a minimal model of the quotient singularity $\bC^n/G$.

To formulate the McKay correspondence and describe the construction of a generating set of the Cox ring of a resolution $X \ra \bC^n/G$ we introduce monomial valuations corresponding to conjugacy classes in~$G$ and the age of a conjugacy class. Fix a primitive root of unity $\zeta_r$ of order $r$. If $g\in G$ has order $r$ then we can diagonalize it so that it takes the form $\diag(\zeta_{r}^{a_1},\ldots,\zeta_{r}^{a_n})$, where $0 \le a_{i} < r$, $a_{i}\in \bZ$. When such a diagonalization is chosen, we may consider a monomial valuation corresponding to~$g$ and compute the age of~$g$.

\begin{definition}
Let $a_1,\ldots,a_n$ be positive integers. The corresponding \textbf{monomial valuation} $\nu$ on $\bC[x_{1},\ldots, x_{n}]$ is
$$\nu \left(\sum_{\alpha}c_{\alpha}x^{\alpha}\right) = \min\{\deg(x^{\alpha})\ : c_{\alpha}\neq 0\}$$
where $\deg(x^{\alpha}) = \sum_{i=1}^n \alpha_ia_i$. It extends uniquely to the field of rational functions. If $\diag(\zeta_{r}^{a_1},\ldots,\zeta_{r}^{a_n})$ is a diagonalization of $g \in G$ as above, we denote the monomial valuation for $a_1,\ldots,a_n$ by~$\nu_g$.
\end{definition}

\begin{definition}
Let $\diag(\zeta_{r}^{a_1},\ldots,\zeta_{r}^{a_n})$ be a diagonalization of $g \in G$. The \textbf{age} of~$g$ is $\age(g) = \frac{1}{r}\sum_{i=1}^{n}a_{i}$. Elements of age~1 will be called \textbf{junior} elements.
\end{definition}

Note that both the monomial valuation and the age of $g\in \SL(n,\bC)$ are invariant under conjugation, but depend on the choice of $\zeta_r$.

By the McKay correspondence we understand the following result of Ito and Reid~\cite{ItoReid}.

\begin{theorem}\label{theorem_mckay}
Let $G\subset \SL(n,\bC)$ be a finite subgroup and let $\varphi:X\to \bC^{n}/G$ be a crepant resolution.
Then there is a one-to-one correspondence between irreducible exceptional divisors of $\varphi$ and conjugacy classes of age~1 in $G$. Moreover, if a divisor~$E$ is sent to the class of element $g$ of order $r$, then the divisorial valuation $\nu_{E}$ is equal to $\frac{1}{r}\nu_{g}|_{\bC(X)}$ (where we identify $\bC(X)$ with $\bC(x_1,\ldots,x_n)^{G}$).
\end{theorem}

\begin{corollary}\label{corollary_mckay}
If $n = 3$, then we can compose the above bijection with the involution sending $g$ to $g^{-1}$ to obtain that irreducible exceptional divisors of~$\varphi$ which are contracted to the point $[0] \in \bC^3/G$ are in one-to-one correspondence with conjugacy classes of age~2 in~$G$.
\end{corollary}

\begin{remark}
Note that though in dimension~3 a crepant resolution for $G \subset \SL(3,\bC)$ always exists, in general it happens that there is no crepant resolution (this is common already for $n=4$). However, Theorem~\ref{theorem_mckay} was proved for any minimal model of a quotient singularity. Such a minimal model always exists, which follows from~\cite{BCHM}.
\end{remark}

We recall an approach to constructing a generating set of a Cox ring of a crepant resolution of a quotient singularity. It was introduced in the case of symplectic singularities and resolutions in~\cite{81resolutions, SymplCox}. Then, in~\cite{Yamagishi} it was noticed that the same is true, and the proofs work in the same way, in the case of any minimal model of a quotient singularity for $G \subset \SL(n,\bC)$, in particular a crepant resolution if it exists. We start from embedding the Cox ring in a bigger ring which is easier to understand.

\begin{proposition}\cite[Prop.~3.8]{81resolutions}\label{prop_structure_cox_ring} Let $\varphi\colon X\ra \bC^{n}/G$ be a crepant resolution for a finite group $G \subset \SL(n,\bC)$. By~\cite[Lem.~2.11]{81resolutions}, $\Cl(X)$ is a free abelian group, so assume $\Cl(X) \simeq \bZ^m$. Then the push-forward of sections induces an embedding
  $$\ovl{\Theta} \colon \cR(X)\subset \cR(\bC^{n}/G)\otimes_{\bC}\bC[\Cl(X)] \simeq \bC[x_{1},\ldots,x_n]^{[G,G]}[t_{1}^{\pm 1},\ldots,t_{m}^{\pm 1}].$$
\end{proposition}

The isomorphism of graded rings $\cR(\bC^{n}/G) \simeq \bC[x_1,\ldots,x_n]^{[G,G]}$ follows by~\cite[Thm~3.1]{AG_finite}; it relies on an observation that $\Cl(\bC^{n}/G) \simeq \Ab(G)^{\vee}$ (see~\cite[Prop.~3.9.3]{Benson}), i.e. grading groups are the same.

We look for generators of the embedded Cox ring. Let $g_1,\ldots, g_m$ be representatives of all junior conjugacy classes in~$G$, and $\nu_{1},\ldots,\nu_{m}$ corresponding monomial valuations. The general procedure is as follows.
\begin{enumerate}
\item Find a minimal generating set $\{\phi_{1},\ldots, \phi_{s}\}$ of $\bC[x_1,\ldots,x_n]^{[G,G]}$, consisting of elements homogeneous with respect to the induced $\Ab(G)$-action.
\item For each $\phi_{i}$ take $\phi_i\cdot t_1^{\nu_1(\phi_i)}\cdots t_m^{\nu_m(\phi_i)}$, and for each variable $t_j$ take $t_j^{-ord(g_j)}$ -- these elements make the candidate for a generating set.
\item Check the condition in Theorem~\ref{theorem_valuation_lifting} for this set. If it is not satisfied then look for additional generators (or for a modification of existing ones which increases valuations) and repeat this step.
\end{enumerate}

Note that in this procedure we choose coordinates $t_1,\ldots, t_m$ in a slightly different way than in~\cite[Sect.~3]{SymplCox}: we let $t_j^{-ord(g_j)}$ correspond to the exceptional divisor~$E_j$.

To formulate the theorem, let $\kappa: \bC[Z_{1},\ldots,Z_{s}]\to \bC[x_1,\ldots,x_n]^{[G,G]}$ be a surjective homomorphism given by sending $Z_{i}$ to $\phi_{i}$. Define monomial valuations $\wt{\nu}_{j}$ on $\bC(Z_{1},\ldots,Z_{s})$ by setting $\wt{\nu}_{j}(Z_i) = \nu_{j}(\phi_{i})$.

\begin{theorem}[{\cite[Thm~3.9]{SymplCox},~\cite[Prop.~4.4]{Yamagishi}}]\label{theorem_valuation_lifting}
The embedded Cox ring $\ovl{\Theta}(\cR(X))$ is generated by $\phi_i\cdot t_1^{\nu_1(\phi_i)}\cdots t_m^{\nu_m(\phi_i)}$ for $i = 1,\ldots,s$ and $t_i^{-ord(g_i)}$ for $i=1,\ldots,m$
if the following \textbf{valuation lifting condition} is satisfied:
for every $f\in \bC[x_1,\ldots,x_n]^{[G,G]}$ homogeneous with respect to the $\Ab(G)$-action there is $\wt{f}\in \kappa^{-1}(f)$ such that $\wt{\nu}_{i}(\wt{f})\ge \nu_i(f)$ for all $i =1,\ldots,m$.
\end{theorem}

We use the valuation lifting condition to determine Cox rings of crepant resolutions of $\bC^{3}/G$ for non-abelian reducible representations $G\subset \SL(3,\bC)$ in section~\ref{section_reducible}, and in section~\ref{section_irreducible} for several examples of irreducible representations. In the reducible case we prove directly that the condition holds for minimal generating set of the ring of invariants of $[G,G]$, using the fact that the ideal of relations $\ker \kappa$ is generated by a single trinomial relation. For irreducible representations we develop a different approach, because the minimal generating set does not usually give a generating set of $\cR(X)$ -- it needs at least linear modifications to increase valuations, but sometimes also adding more generators. We verify that obtained elements generate $\cR(X)$ by applying procedures in the library developed for~\cite{quotsingcox} or by the algorithm from~\cite{Yamagishi}.

\subsection{Cone of movable divisors and its chamber decomposition}
\label{section_methods_Mov}
Assume that a variety~$X$ satisfies conditions listed at the beginning of section~\ref{section_methods_Cox}. Here we explain how to compute the decomposition of the movable cone $\Mov(X)$ of~$X$ from the ideal of relations of its finitely generated Cox ring $\cR(X)$ and the matrix of the Picard torus $\bT_X$ action on $\Spec \cR(X)$.  We apply this to determine the number of crepant resolutions of investigated 3-dimensional quotient singularities, the relations, i.e. sequences of flops, between these resolutions and the structure of their exceptional sets (in chosen cases). We rely on the fact that crepant resolutions of $\bC^3/G$ are in one-to-one correspondence with Mori chambers in their common cone of movable divisors. This is because all projective crepant resolutions are connected by a sequence of flops (as minimal models of the singularity), and~\cite[Thm~2.4]{KollarFlops} implies that in dimension~3 flops preserve smoothness.

We start from a fixed generating set $\{\phi_1,\ldots,\phi_s\}$ of the Cox ring $\cR(X)$. Consider the ring homomorphism $\bC[Z_{1},\ldots,Z_{s}] \ra \cR(X)$ sending $Z_i$ to $\phi_i$ and the associated embedding $\Spec \cR (X) \hookrightarrow \bC^s$. The Picard torus $\bT_X$ of rank~$m$ acts on $\Spec \cR(X)$ as a subtorus of the big torus $(\bC^*)^s$ of $\bC^s$. The matrix $U$ of the weights of this action defines a lattice homomorphism $\bZ^s \ra \bZ^{m} \simeq \Cl(X)$ between character lattices of $(\bC^*)^s$ and $\bT_X$. We will apply the induced map of vector spaces $U \colon \bR^s \ra \bR^m$ to rational polyhedral cones. By $\gamma$ we denote the positive orthant of $\bR^s$. Then there is a simple description of the cone of movable divisors.

\begin{proposition}{\cite[Prop.~3.3.2.9]{CoxRings}}\label{prop_mov_intersection}
The cone $\Mov(X)$ is the intersection of the images of all facets of $\gamma$ under $U$.
\end{proposition}

To compute its chamber subdivision, we use the terminology introduced in~\cite{BerchtoldHausenGIT} and, in the algorithmic setting, in~\cite{KeicherGITFan} (see also~\cite[Sect.~4]{81resolutions} for an explanation in the context of resolutions of quotient singularities). Their approach is based on using the orbit decomposition of $\bC^s$ under its big torus action. It commutes with the action of $\bT_X$, hence certain properties (in particular semistability) depend only on the choice of the orbit of $(\bC^*)^s$.

\begin{remark}
Recall that there is a one-to-one correspondence between faces of~$\gamma$ and orbits of the action of $(\bC^*)^s$ on $\bC^s$: a face spanned by vectors $e_{i_1},\ldots,e_{i_k}$ of the standard basis gives an orbit of points where precisely the coordinates $i_1,\ldots,i_k$ are non-zero.
\end{remark}

Let~$I$ be the ideal of $\Spec \cR(X) \subset \bC^s$.

\begin{definition}\label{def_aface}
An \textbf{$I$-face} of $\gamma$ is a face such that the corresponding $(\bC^*)^s$-orbit of $\bC^s$ intersects $\Spec \cR(X)$. A \textbf{projected face} (or an \textbf{orbit cone}) is the image of a face of $\gamma$ under~$U$.
\end{definition}

\begin{proposition}\cite[Prop.~2.9]{BerchtoldHausenGIT}, see also~\cite[Prop.~3.1.1.12]{CoxRings}.
The subdivision of $\Mov(X)$ into GIT chambers is given by intersections of all projected $I$-faces. More precisely, the chamber containing $w\in \Mov(X)$ in its interior is the intersection of all projected $I$-faces containing~$w$ in its interior.
\end{proposition}

There is a tool for computing projected $I$-faces and the full GIT-fan (which we just need to restrict to $\Mov(X)$ for our purposes) for a given embedding of a variety: the library \texttt{gitfan.lib} written for Singular~\cite{Singular} by S.~Keicher, see~\cite{KeicherGITFan}.

To determine the structure of the central fibre (i.e. the fibre over $[0] \in \bC^3/G$) of a resolution we need to fix a linearisation $w$ from the interior of the corresponding GIT chamber and compute all $I$-faces stable with respect to~$w$. This, again, can be reduced to a purely combinatorial procedure. It is described in~\cite[Sect.~4]{81resolutions}, we perform the computations using a Singular package developed for that project.

\section{Dihedral groups}\label{section_dihedral}

Let $G = D_{2n}$, $n \ge 3$,  be a dihedral group of order $2n$, i.e. the group of isometries of the regular $n$-gon in plane. In terms of generators and relations $G$ can be presented as $\langle \rho, \eps \ | \ \rho^n,\ \eps^2,\ (\eps\rho)^2 \rangle$, and its elements are 1, rotations $\rho,\rho^2,\ldots, \rho^{n-1}$ and reflections $\eps, \eps\rho, \ldots, \eps\rho^{n-1}$. The structure of the sets of conjugacy classes of~$G$, which are very important for the properties of the resolution because of the McKay correspondence, differs depending on the parity of~$n$ -- this is why we describe these cases separately.

We consider the following 3-dimensional representation of~$G$:
\begin{equation}\label{gens_dihedral}
\rho \mapsto \begin{pmatrix}
\zeta & 0 & 0\\
0 & \zeta^{-1} & 0\\
0 & 0 & 1
\end{pmatrix}, \qquad 
\eps \mapsto \begin{pmatrix}
0 & 1 & 0\\
1 & 0 & 0\\
0 & 0 & -1
\end{pmatrix},
\end{equation}
where $\zeta$ denotes the primitive $n$-th root of unity. It is easy to check that the image of any other faithful representation of $G$ in $\SL(3,\bC)$ is the same subgroup up to conjugacy. By abuse of notation, from now on we will denote by $G$ the image of given representation and by $\rho,\varepsilon$ their images. Note that this is a representation without quasi-reflections (matrices with all but one eigenvalues~1).

\subsection{The odd case: $n = 2k+1$} The commutator subgroup consists of all rotations: $[G,G] = \langle \rho \rangle$. All reflections are conjugate and pairs of rotations are conjugate, so there are $k+2$ conjugacy classes:
$$\{1\},\:\{\rho,\rho^{-1}\},\ldots, \{\rho^{k},\rho^{-k}\},\: \{\eps,\eps\rho,\eps\rho^2,\ldots,\eps\rho^{2k}\}.$$

The set of points in $\bC^3$ with nontrivial isotropy group consists of the line $x_1=x_2=0$ fixed by $\langle \rho \rangle$ and $n$ lines, each fixed by a reflection, e.g. $x_1-x_2=x_3=0$ fixed by $\langle \eps \rangle$. In the quotient $\bC^3/G$ these lines are mapped to two components of the singular points set: the first one to a component $L_{\rho}$ with transversal $A_{2k}$ singularity and lines fixed by reflections to a component $L_{\eps}$ with transversal $A_1$ singularity (away from 0). The image of 0, as the intersection point of these components, has a worse singularity.

From the McKay correspondence we obtain that each nontrivial conjugacy class in~$G$ correspond to an exceptional divisor in the resolution which is mapped to a line of singular points in $\bC^3/G$. This is because there are no elements of age~2, which would give an exceptional divisor in the fibre over the origin. Thus we have $k$ exceptional divisors $E_1,\ldots, E_k$ mapped to $L_{\rho}$ and $E_{\eps}$ mapped to $L_{\eps}$.

\subsubsection{The Cox ring}
The Cox ring $\cR(\bC^3/G)$ is the ring of invariants $\bC[x,y,z]^{[G,G]}$ of the commutator subgroup $[G,G] = \langle \rho\rangle$. To find the Cox ring of a crepant resolution we also have to determine the eigenspaces of the action of the abelianization $\Ab(G) \simeq \bZ_2$, generated by the class of $\eps$, on $\cR(\bC^3/G)$.

\begin{lemma}
$\cR(\bC^3/G) = \bC[x,y,z]^{[G,G]}$  is generated by $x^{n}+y^{n}$, $x^{n} - y^{n}$, $xy$, $z$, where $x^{n}+y^{n}$ and $xy$ are also $\eps$-invariant, but on $x^{n}-y^{n}$ and $z$ the reflection $\eps$ acts by multiplying by~$-1$.
\end{lemma}

Next, we need the values of monomial valuations corresponding to all conjugacy classes (of age~1) in $G$.

\begin{lemma}\label{table:valuationsOdd} The values of monomial valuations on given generators of $\cR(\bC^3/G)$ are as follows:
\begin{center}
\begin{tabular}{c|cccc}
val\textbackslash gen & $x^n+y^n$ & $x^n-y^n$ & $xy$ & $z$ \\
\hline
$\nu_{\rho^{i}}$ & $\frac{in}{\gcd(n,i)}$ & $\frac{in}{\gcd(n,i)}$ & $\frac{n}{\gcd(n,i)}$ & 0 \\
$\nu_{\eps}$ & 0 & 1 & 0 & 1
\end{tabular}
\end{center}
\end{lemma}

Let $\varphi:X\to \bC^3/G$ be a crepant resolution.

\begin{theorem}\label{theorem:Cox-odd-case}
The Cox ring $\cR(X)$ as a $\bC$-subalgebra of $\bC[x,y,z]^{[G,G]}[t_{\eps}^{\pm 1}, t_{i}^{\pm 1} \colon i \in \{1,\ldots,k\}]$ is generated by
\begin{multline*}
(x^{n}+y^{n})\prod_{i=1}^{k}t_{i}^{\frac{in}{\gcd(n,i)}}, \  (x^{n} - y^{n})t_{\eps}\prod_{i=1}^{k}t_{i}^{\frac{in}{\gcd(n,i)}}, \ xy\prod_{i=1}^{k}t_{i}^{\frac{n}{\gcd(n,i)}},\ zt_{\eps},\\ t_{\eps}^{-2},\ \{t_{i}^{-\frac{n}{\gcd(n,i)}} \colon i \in \{1,\ldots,k\}\}. 
\end{multline*} 
\end{theorem}
\begin{proof}
It follows from a more general result, Theorem~\ref{theorem:cox-rings-reducible}, but one can also give a proof based on the characterisation of the Cox ring via geometric quotients~\cite[1.6.4.3]{CoxRings} (as in the surface case in~\cite{CoxSurf}), which requires in particular a direct check that quotients of $\Spec \cR(X)$ are smooth.
\end{proof}

We define the surjective map
\begin{equation}\label{map_pi_odd}
  \kappa:\bC[Z_{1},Z_{2},Z_{3},Z_{4}, T_{\eps}, T_{i} \colon i \in \{1,\ldots,k\}]\longrightarrow \cR(X)
\end{equation}
which sends each variable to the respective generator of $\cR(X)$.
\begin{corollary}\label{proposition:relations-odd-case}
The ideal $I_k = \ker \kappa$, corresponding to the embedding of $\Spec\cR(X)$ in~$\bC^{k+5}$, is generated by a single trinomial
\begin{equation}\label{equation_cox_dn_odd}
Z_{1}^{2} - Z_{2}^{2}T_{\eps} - 4Z_{3}^{n}\prod_{i=1}^{k}T_{i}^{n-2i}.
\end{equation}
\end{corollary}

\subsubsection{The Mov cone and the set of all crepant resolutions}\label{dihedral_odd_mov}
The monomial valuations from Lemma~\ref{table:valuationsOdd} describe the action of the Picard torus on $\Spec \cR(X)$ -- the matrix of the action is:

\begin{equation*}
\left( \begin{array}{c|cc|ccccc|c}
\frac{n}{\gcd(n,1)} &   \frac{n}{\gcd(n,1)} & 0 & \frac{n}{\gcd(n,1)} & -\frac{n}{\gcd(n,1)} & 0 &  & 0 & 0\\
\frac{2n}{\gcd(n,2)} &   \frac{2n}{\gcd(n,2)} & 0 & \frac{n}{\gcd(n,2)} & 0 & -\frac{n}{\gcd(n,2)}&  & 0 & 0 \\
\vdots &   \vdots & \vdots & \vdots & \vdots &\vdots & \ddots & \vdots & \vdots \\
\frac{kn}{\gcd(n,k)} &   \frac{kn}{\gcd(n,k)}& 0 & \frac{n}{\gcd(n,k)} & 0 & 0 &  & -\frac{n}{\gcd(n,k)} & 0\\
0 & 1 & -2 & 0 & 0 & 0 &  & 0 & 1 \\
\end{array}
\right)
\end{equation*}

The first three groups of columns correspond to sets of variables in monomials in the trinomial relation defining $\Spec \cR (X)$ in Corollary~\ref{proposition:relations-odd-case}, and the last one corresponds to $Z_{4}$ not involved in the relation. 

We change the coordinates in the image $\bR^{k+1}$ by rescaling the $i$-th one by $\frac{\gcd(n,i)}{n}$ for $i \in \{1,\ldots,k\}$ -- this does not change the combinatorial properties of the $\Mov(X)$ and its chambers. Thus we obtain
\begin{equation*}
U_k := \left( \begin{array}{c|cc|ccccc|c}
1 & 1 & 0 & 1 & -1 & 0 &  & 0 & 0\\
2 & 2 & 0 & 1 & 0 & -1&  & 0 & 0 \\
\vdots &   \vdots & \vdots & \vdots & \vdots &\vdots & \ddots & \vdots & \vdots \\
k & k & 0 & 1 & 0 & 0 &  & -1 & 0\\
0 & 1 & -2 & 0 & 0 & 0 &  & 0 & 1 \\
\end{array}
\right)
\end{equation*}

We are ready to describe the cone $\Mov(X)$ in $N^{1}(X) = \Cl(X)\otimes_{\bZ}\bR\simeq \bR^{k+1}$. We denote the standard basis by $e_1,\ldots,e_{k+1}$. Note that in these coordinates the exceptional divisors are $E_i = -e_i$ for $i \in \{1,\ldots,k\}$ and $E_{\eps} = -2e_{k+1}$. We also use the following vectors:
\begin{align}\label{eq_important_vec}
  \begin{split}
  &q = e_{k+1} = (0,\ldots,0,1),\\
  &v_{i} = (1,2,\ldots,i-1, i,\ldots, i,0)\quad \hbox{for } i \in \{1,\ldots,k\},\\
    &w_{i} = (1,2,\ldots, i,\ldots, i,1) = v_{i}+q\quad  \hbox{ for } i \in \{1,\ldots,k\}.
  \end{split}
\end{align}

\begin{proposition}\label{proposition:mov-odd}
The movable cone $\Mov(X)$ is spanned by rays $q,v_1,\ldots,v_k$ and defined by inequalities
\begin{equation*}\label{eqn:mov-inequalities-odd}
\{(a_{1},\ldots,a_{k+1})\ :\ a_{k+1}\ge 0,\ 2a_{1} \ge a_{2},\ a_{k}\ge a_{k-1},\ 2a_{i}\ge a_{i-1}+a_{i+1},\ \ 1< i < k \}.
\end{equation*}
\end{proposition}

\begin{proof}
First, it is easy to check that the cones defined by given rays and by given inequalities are equal: obviously all rays satisfy inequalities, and one can construct inductively a positive combination of rays equal to a vector satisfying inequalities. We denote this cone by $C_k$.

Then by Proposition~\ref{prop_mov_intersection} we have to show that $C_k$ is equal to $\Mov(X)$ defined as the intersection of the images of facets of the positive orthant $\gamma \subset \bR^{k+5}$ under $U_k$. That is, we describe $\Mov(X)$ as the intersection of cones spanned by any set of $k+4$ columns of~$U_k$. To prove $C_k \subset \Mov(X)$ we look at the rays of $C_k$ and show that they belong to the image of any facet. And for $C_k \supset \Mov(X)$ we use the inequalities for $C_k$: for each of them one can find a facet whose image satisfies considered inequality. We skip the details of this purely combinatorial argument.
\end{proof}

\begin{proposition}\label{proposition:GIT-subdivision-odd}
There are $k+1$ GIT chambers of $\Mov(X)$:
\begin{equation*}
\sigma_0 = cone(q,w_1,\ldots,w_k),\, \sigma_i = cone(v_{1},v_2,\ldots, v_{i},w_{i},w_{i+1},\ldots,w_{k}) \hbox{ for } i \in \{1,\ldots,k\}
\end{equation*}
where $q$, $v_i$ and $w_i$ are as in~\eqref{eq_important_vec}.
\end{proposition}

To give the proof we need to describe projected $I_k$-faces of the positive orthant $\gamma \subset \bR^{k+5}$ under $U_k$ which are big enough to cut out a full-dimensional chamber.

\begin{lemma}\label{lemma:maximal-a-faces-odd}
The projected $I_k$-faces of dimension $k+1$ are spanned by the sets of rays given below. Note that indices $i \leq j$ (possibly equal!) are always from $\{1,\ldots,k\}$.
\begin{align}
\begin{split}\label{proj_faces_1}
  &\{v_1,- e_{1},\ldots,-e_k,\pm e_{k+1}\}, \{v_1,- e_{1},\ldots, - e_{k},e_{k+1}\}, \{v_1,- e_{1},\ldots, -e_{k+1}\}, \\
\end{split}\\
\begin{split}\label{proj_faces_2}
& \{q,v_1,w_k, -e_{1},\ldots, -e_{k+1} \} \setminus \{-e_i, -e_j\},\,\, \{q,w_k, -e_1,\ldots, -e_{k+1} \} \setminus \{-e_i\},\\
& \{v_{1},w_{k}, -e_1,\ldots, -e_{k+1} \} \setminus \{-e_i, -e_j\},\quad \{w_{k},-e_1,\ldots, -e_{k+1} \} \setminus \{-e_i\},\\
\end{split}\\
\begin{split}\label{proj_faces_3}
  & \{-e_{1},\ldots,-e_{k},\pm e_{k+1}\},\quad \{-e_{1},\ldots,-e_{k},-e_{k+1}\},\quad \{-e_{1},\ldots,-e_{k}, e_{k+1}\},\\
\end{split}\\
\begin{split}\label{proj_faces_4}
  & \{q,v_{1},w_{k},-e_{1},\ldots,-e_k\} \setminus \{-e_i, -e_j\},\quad \{q,w_{k},-e_{1},\ldots,-e_k \} \setminus \{-e_i\},\\
\end{split}\\
\begin{split}\label{proj_faces_5}
& \{v_{1},-e_{1},\ldots, -e_{k+1} \} \setminus \{-e_i\},\quad \{q,v_{1},-e_{1},\ldots, -e_{k+1} \} \setminus \{-e_i\}.
  \end{split}
\end{align}
\end{lemma}

\begin{proof}
  By definition~\ref{def_aface} we need to choose all sets of columns of $U_k$ which span~$\bR^{k+1}$ such that if we set all variables corresponding to \emph{not}~chosen columns to 0 in the trinomial~\eqref{equation_cox_dn_odd} then we still can find non-zero values for the remaining variables to make the trinomial vanish. This last condition means that we may set to 0 precisely 3, 1 or 0 monomials in the trinomial (by choosing columns corresponding to any sets of their variables), but not 2 of them. In this way we make the list of all possible sets given above, checking if they span~$\bR^k$, simplifying them and avoiding redundancies.

  If the third monomial $Z_3\prod_{i}T_i$ is non-zero, then we can get only the whole space or a halfspace,~\eqref{proj_faces_1}. Four sets in~\eqref{proj_faces_2} correspond to setting only the third monomial to~0: note that we may remove at most 2 columns from the third group to have a set spanning~$\bR^{k+1}$. The last three cases describe the situation with all three monomials set to 0. In~\eqref{proj_faces_3} we have $Z_1 = Z_2 = Z_3 = 0$, in~\eqref{proj_faces_4} $Z_1 = T_{\eps} = Z_3\prod_{i}T_i = 0$ (note that here we need the last column to span the whole~$\bR^{k+1}$) and in~\eqref{proj_faces_5} $Z_1 = Z_2 = \prod_i T_i = 0$.
  
\end{proof}

\begin{proof}[Proof of Proposition~\ref{proposition:GIT-subdivision-odd}]
First note that $\Mov(X)$ is indeed a sum of given chambers. All chambers are contained in $\Mov(X)$ since all their rays lie in this cone. For the other inclusion, take $v = aq + a_{1}v_{1}+\ldots+a_kv_k\in \Mov(X)$. If $a \ge a_1+\ldots + a_k$ then $v$ is in $\sigma_0$:
\begin{equation*}
v = (a-a_{1}-\ldots-a_k)q + a_{1}w_{1}+\ldots+a_kw_k.
\end{equation*}
Otherwise, let $j$ be maximal such that $a< a_{j}+\ldots+a_k$. Let $S_j = a_j+\ldots + a_k$. Then $v$ can be written as an element of $\sigma_j$:
\begin{equation*}
v = a_{1}v_1+\ldots+a_{j-1}v_{j-1} + (S_j-a)v_j+ (a-S_{j+1})w_j+ a_{j+1}w_{j+1}+\ldots+a_k w_k.
\end{equation*}

Next, we prove that projected $I_k$-faces do not subdivide given chambers. That is, any projected $I_k$-face of maximal dimension either contains a whole chamber, or does not contain any of its interior points. It can be checked directly, by showing for each chamber and each type of $I_k$-face of maximal dimension, that either all rays of the chamber lie within the $I_k$-face or that there is a hyperplane separating interiors of these cones. An easy way for finding such a hyperplane is to determine it by computation for small $k$ and then deduce the equation in the general case. Since the results are similar for all types of $I_k$-faces, we present here just a single case, one of the least trivial.

Consider $I_k$-faces of the first type from~\eqref{proj_faces_4}. If $j=i$ or $j = i+1$ then such an $I_k$-face contains chambers $\sigma_0,\ldots, \sigma_i$, because one can construct $v_1,\ldots,v_i$ from its rays as combinations of $mv_1$ and negatives of basis vectors, and $w_1,\ldots, w_k$ either as $v_i+q$ or as combinations of $w_k$ and negatives of basis vectors. The hyperplane $a_{i+1}-a_i = a_{k+1}$ separates interiors of such an $I_k$-face and chambers $\sigma_{i+1},\ldots, \sigma_k$. Now assume that $j-i > 1$. Then the hyperplane $a_i+a_j = a_{i+1} + a_{j-1}$ separates the interiors of the $I_k$-face and any of the chambers (i.e. the whole $\Mov(X)$).

Finally, we have to show that any chamber can be separated from any other using a projected $I_k$-face. But this follows from the previous paragraph: $I_k$-faces of the first type from~\eqref{proj_faces_4} separate $\sigma_i$ from $\sigma_{i+1},\ldots,\sigma_k$ for $i = 0,\ldots,k-1$.
\end{proof}

\begin{example}\label{example_D7}
  For $D_{14}$, that is $k = 3$, we give the equation for the Cox ring of crepant resolutions and the matrix of the Picard torus action. The picture shows a 3-dimensional section of the $\Mov$ cone subdivided into chambers. Three chambers, sharing an edge on the back face of the tetrahedron, are shown in grey, the last one is white. One sees that there is only one way of walking through chambers, starting from the leftmost or the topmost, i.e. crepant resolutions and flops form a sequence.
  
  \begin{tabular}{cr}
 & \multirow{4}{*}{\includegraphics[width=0.5\textwidth]{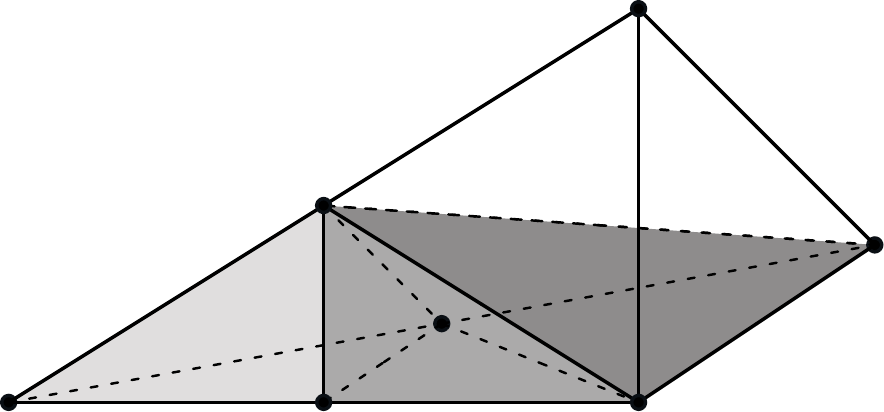}} \\
$\left( \begin{array}{c|cc|cccc|c}
1 & 1 & 0  & 1 & -1 & 0 & 0 & 0 \\
2 & 2 & 0  & 1 & 0 & -1 & 0 & 0  \\
3 & 3 & 0  & 1 & 0 & 0 & -1 & 0  \\
0 & 1 & -2 & 0 & 0 & 0 & 0 & 1  \\
\end{array}
  \right)$
  &  \\
  & \\
$Z_{1}^{2} - Z_{2}^{2}T_{\eps} - 4Z_3^{7}T_{1}^{5}T_{2}^{3}T_{3}$ &  
\end{tabular}

\end{example}

\subsubsection{Structure of the central fibre}\label{section_central_fibre_odd}

We describe the structure of the central fibre of the resolution, check how it changes under a flop and find the chamber corresponding to the $G$-$\Hilb$ resolution.

\begin{proposition}\label{propostion:dihedral-invariants}
If $G$ is a representation of a dihedral group defined as in~\eqref{gens_dihedral} then $\bC[x,y,z]^{G} = \bC[x^{n}+y^{n}, (x^{n}-y^{n})z, xy,z^2]$.
\end{proposition}

\begin{proof}
We take a generating set of the invariants of the first (diagonal) generator~$\rho$ and modify it to obtain eigenvectors of the action of the second generator~$\eps$. (In the case of~$n$ odd we have already seen this, since~$\rho$ generates~$[G,G]$ and~$\eps$ determines the action of~$\Ab(G)$.) Thus we have an induced diagonal action of~$\eps$ on the invariant ring of~$\rho$, so we just need to determine invariant monomials in generators of this ring.
\end{proof}

The ideal of $\bC[Z_{1},Z_{2},Z_{3},Z_{4}, T_{\eps}, T_{i} \colon i \in \{1,\ldots,k\}]$ of the subset of $\Spec \cR(X)$ which is mapped to the central fibre of the resolution by the GIT quotient morphism is the inverse image under $\kappa$, see~\eqref{map_pi_odd}, of the ideal generated by the non-constant elements of~$\bC[x,y,z]^{G}$ in~$\cR(X)$. (This is the same as the ideal generated by~$I_k$ and the Picard torus invariants.)

\begin{corollary}
The ideal of the subset of $\Spec \cR(X) \subset \bC^{k+5}$ mapped to the central fibre of $X/G$ is
$$J_k = I_k + (Z_1T_{1}T_2^2\cdots T_{k}^{k},\, Z_{2}Z_{4}T_{\eps}T_{1}T_2^2\cdots T_{k}^{k},\, Z_{3}T_{1}T_2\cdots T_{k},\, Z_4^2 T_{\eps}).$$  
\end{corollary}

\begin{lemma}
The subset of $\Spec \cR(X) \subset \bC^{k+5}$ mapped to the central fibre of the resolution $X$ decomposes into the following irreducible components:
\begin{align*}
  & W_i = \{Z_{4} = T_i = Z_{1}^{2} - Z_{2}^{2}T_{\eps} = 0\}\hbox{ and } W'_i = \{Z_{1} = T_{\eps} = T_{i} = 0\} \hbox{ for } i \in \{1,\ldots,k\},\\
  & W_{0} = \{Z_{1} = Z_{3} =  T_{\eps} = 0\},\quad W_{u} = \{Z_{1} = Z_{2} = Z_{3} = Z_{4} = 0\}.
\end{align*}
\end{lemma}

\begin{proof}
One can prove directly that the radical of $J_k$ is equal to the intersection of all ideals listed above; the key observation is that $Z_{1}^{2} - Z_{2}^{2}T_{\eps}$ is in the radical of~$J_k$.
\end{proof}

Let $\varphi_i\colon X_i\to \bC^{3}/G$ be the resolution corresponding to the choice of the linearisation of the Picard torus action given by a character from the interior of the chamber~$\sigma_i$.

\begin{proposition}
For $X_i$ the  stable components of the set mapped to the central fibre are $W_1,\ldots, W_i,W_i',\ldots, W_k'$.
For $X_0$ the stable components are $W_{0}, W'_1, \ldots, W'_k$.
\end{proposition}

\begin{proof}
  We show that $W_j'$ for $j = 1,\ldots, i-1$ are not stable on $X_i$; the remaining assertions are proved similarly. To fix a linearisation corresponding to $X_i$ we pick a vector $s_i$ from $\sigma_i$ which is the sum of its rays. We need to show that it does not belong to the interior of the orbit cone $C_j'$ corresponding to an open subset of~$W_j'$. The cone $C_j'$ is spanned by all columns of $U_k$ except the ones corresponding to $Z_1$, $T_{\eps}$ and $T_j$.

  Take a positive combination of rays of $C_j'$ with coefficient $\alpha$ at $(1,2,\ldots,k,1)$ and $\beta$ at $(1,\ldots,1,0)$, and assume it gives~$s_i$. From a direct computation of coordinates of~$s_i$ ($j$, $j+1$ and $k+1$ respectively) and the negativity of certain rays of $C_j'$ we get three conditions:
\begin{align*}
  j\alpha + \beta &= (k(k+1) - (k-j)(k-j+1))/2+j,\\
  (j+1)\alpha + \beta &\geq (k(k+1) - (k-j-1)(k-j))/2+j+1,\\
  \alpha &\leq k-j+1.
\end{align*}
This leads to $k-i+1 \geq \alpha \geq k-j+1$, which contradicts the assumption $j < i$.

The investigation of stability of $W_i$ relies on the fact that its open subset corresponds to an orbit cone spanned by all columns of~$U_k$ except these corresponding to~$Z_4$ and~$T_i$. Finally, note that in a similar way one could list all orbit cones corresponding to stable points of $\Spec \cR(X)$, which we skip for the sake of brevity.
\end{proof}

Using similar methods to investigate the stability of the intersections of components $W_0,W_1,\ldots,W_k,W_1',\ldots,W_k'$ one also checks that the components of the central fibre form a chain.

\begin{corollary}
By analysing the proof of Theorem~5.1 in~\cite{NdCS} one checks that the $G$-$\Hilb$ resolution corresponds to the chamber $\sigma_0$.
\end{corollary}

\subsection{The even case: $n = 2k$} Since we have presented the details of our method in the case of $D_{2n}$ for odd $n$, here we skip some details and arguments if they are the same as in the previous case.

The commutator subgroup consists of all even rotations: $[G,G] = \langle \rho^2 \rangle$. Then $\Ab(G) \simeq \bZ_2\times \bZ_2$ and we use classes of $\eps$ and $\eps\rho$ as its generators. Pairs of mutually inverse rotations are conjugate and reflections make two conjugacy classes, so there are $k+3$ conjugacy classes:
$$\{1\},\:\{\rho,\rho^{-1}\},\ldots, \{\rho^{k-1},\rho^{-k+1}\},\: \{\rho^{k}\},\: \{\eps,\eps\rho^2,\ldots,\eps\rho^{2k-2}\},\: \{\eps\rho,\eps\rho^3,\ldots, \eps\rho^{2k-1}\}.$$

The set of points in $\bC^3$ with nontrivial isotropy group consists of the line $x_1=x_2=0$ fixed by $\langle \rho \rangle$ and $n$ lines, each fixed by a reflection, e.g. $x_1-x_2=x_3=0$ fixed by $\langle \eps \rangle$ and $x_1 - \zeta x_2 = x_3 = 0$ fixed by $\langle \eps\rho \rangle$. In the quotient $\bC^3/G$ these lines are mapped to three components of the singular points set: a component $L_{\rho}$ with transversal $A_{2k-1}$ singularity and two components $L_{\eps}$ and $L_{\eps\rho}$ with transversal $A_1$ singularity (away from 0). The image of 0 has a worse singularity.

By the McKay correspondence we have $k$ exceptional divisors $E_1,\ldots, E_k$ mapped to $L_{\rho}$ and $E_{\eps}, E_{\eps\rho}$ mapped to $L_{\eps}, L_{\eps\rho}$ respectively.

\subsubsection{The Cox ring}
We compute the ring of $[G,G]$-invariants, give a generating set of eigenvectors of the $\Ab(G)$ action and provide values of all monomial valuations corresponding to conjugacy classes (of age~1) on these generators.

\begin{lemma}
  $\cR(\bC^3/G) = \bC[x,y,z]^{[G,G]}$  is generated by $x^k+y^k$, $x^k-y^k$, $xy$, $z$, where $\eps$ acts trivially on $x^k+y^k$ and $xy$ and multiplies by $-1$ the remaining two generators, and $\eps\rho$ acts trivially on $x^k-y^k$ and $z$ and multiplies by $-1$ the remaining two generators.
\end{lemma}

\begin{lemma}\label{table:valuationsEven} The values of monomial valuations on given generators of $\cR(\bC^3/G)$ are as follows:
\begin{center}
\begin{tabular}{c|cccc}
val\textbackslash gen & $x^k+y^k$ & $x^k-y^k$ & $xy$ & $z$ \\
\hline
$\nu_{\rho^{i}}$ & $\frac{ik}{\gcd(n,i)}$ & $\frac{ik}{\gcd(n,i)}$ & $\frac{n}{\gcd(n,i)}$ & 0 \\
$\nu_{\eps}$ & 0 & 1 & 0 & 1\\
$\nu_{\eps\rho}$ & 1 & 0 & 0 & 1
\end{tabular}
\end{center}
\end{lemma}

Let $\varphi:X\to \bC^3/G$ be a crepant resolution.

\begin{theorem}\label{theorem:Cox-even-case}
The Cox ring $\cR(X)$ as a $\bC$-subalgebra of $\bC[x,y,z]^{[G,G]}[t_{\eps}^{\pm 1}, t_{\eps\rho}^{\pm 1}, t_{i}^{\pm 1}\colon i \in \{1,\ldots,k\} ]$ is generated by:
\begin{multline*}
  (x^{k}+y^{k})t_{\eps\rho}\prod_{i=1}^{k}t_{i}^{\frac{ik}{\gcd(n,i)}}, \  (x^{k} - y^{k})t_{\eps}\prod_{i=1}^{k}t_{i}^{\frac{ik}{\gcd(n,i)}}, \ xy\prod_{i=1}^{k}t_{i}^{\frac{n}{\gcd(n,i)}}, \ zt_{\eps}t_{\eps\rho},\\
t_{\eps}^{-2}, \ t_{\eps\rho}^{-2}, \ \{t_{i}^{-\frac{n}{\gcd(n,i)}} \colon i \in \{1,\ldots, k\}\}.
\end{multline*} 
\end{theorem}

We define the surjective map
\begin{equation}\label{map_pi_even}
  \kappa:\bC[Z_{1},Z_{2},Z_{3},Z_{4}, T_{\eps}, T_{\eps\rho}, T_{i} \colon i \in \{1,\ldots,k\}]\longrightarrow \cR(X)
\end{equation}
which sends each variable to the respective generator of $\cR(X)$.

\begin{corollary}\label{dihedral_even_eq}
The ideal $I_k = \ker \kappa$ of $\Spec\cR(X) \subset \bC^{k+6}$ is generated~by
\begin{equation*}
Z_{1}^{2}T_{\eps\rho} - Z_{2}^{2}T_{\eps} - 4Z_{3}^{k}\prod_{i=1}^{k-1}T_{i}^{k-i}.
\end{equation*}
\end{corollary}

\subsubsection{The Mov cone and the set of all crepant resolutions}
The matrix of weights of the Picard torus action on $\Spec \cR(X) \subset \bC^6$, given by values of monomial valuations listed above, subdivided into groups of columns corresponding to monomials in the triomial defining $\Spec \cR(X)$ and rescaled as in section~\ref{dihedral_odd_mov}, is

\begin{equation*}
U_k = \left( \begin{array}{cc|cc|ccccc|cc}
1 & 0 & 1 & 0 & 2 & -2 & 0 &  & 0 & 0 & 0\\
2 & 0 & 2 & 0 & 2 & 0 & -2 &  & 0 & 0 & 0 \\
\vdots & \vdots & \vdots & \vdots & \vdots & \vdots &\vdots & \ddots & \vdots & \vdots & \vdots \\
k-1 & 0 & k-1 & 0 & 2 & 0 & 0 &  & -2 & 0 & 0\\
k & 0 & k & 0 & 2 & 0 & 0 &  & 0 & -2 & 0\\
0 & 0 & 1 & -2 & 0 & 0 & 0 &  & 0 & 0 & 1 \\
1 & -2 & 0 & 0 & 0 & 0 & 0 &  & 0 & 0 & 1
\end{array}
\right)
\end{equation*}

Now we describe $\Mov(X) \subset N^{1}(X) = \Cl(X)\otimes_{\bZ}\bR\simeq \bR^{k+2}$. We also use the following vectors:
\begin{align}\label{eq_important_vec_even}
  \begin{split}
  &q_1 = (0,\ldots,0,1,0),\qquad q_{2} = (0,\ldots,0,0,1), \qquad q_{3} = (0,\ldots,0,1,1)\\
    &u_{i} = (1,2,\ldots, i-1,i,\ldots,i,0,0) \hbox{ for } i\in \{1,\ldots,k\},\\
    &v_{i}=u_{i}+q_1 = (1,2,\ldots, i-1,i,\ldots, i,1,0) \hbox{ for } i\in \{1,\ldots,k\},\\
&v_{i}'=u_{i} + q_2 = (1,2,\ldots, i-1,i,\ldots, i,0,1) \hbox{ for } i\in \{1,\ldots,k\},\\
&w_{i}  = 2u_{i}  + q_{3} = (2,4,\ldots, 2i-2, 2i,\ldots, 2i, 1,1) \hbox{ for } i\in \{1,\ldots,k\}.
 \end{split}
\end{align}

The next two statements can be proved based on the same ideas as for $n$ odd, see~\ref{proposition:mov-odd} and~\ref{proposition:GIT-subdivision-odd}, but they require more cases to check, hence we skip the details.

\begin{proposition}\label{proposition:mov-even}
The movable cone $\Mov(X)$ is spanned by rays $q_1, q_2, u_1,\ldots,u_k$ and defined by inequalities
\begin{equation*}\label{eqn:mov-inequalities-even}
\{(a_{1},\ldots,a_{k+2})\ : \ a_{k+1}\ge 0,\ a_{k+2}\ge 0, \ 2a_{1}\ge a_{2},\ a_{k}\ge a_{k-1},\ 2a_{i}\ge a_{i-1}+a_{i+1}, \ 1< i< k\}.
\end{equation*}
\end{proposition}

\begin{proposition}\label{proposition:GIT-subdivision-even}
There are $(k+1)^2$ GIT chambers of $\Mov(X)$:
\begin{align*}
  & cone(q_1,q_3,v_1,\ldots,v_k),\ cone(q_2,q_3,v_{1}',\ldots,v_{k}'),\\
  & cone(u_1,\ldots,u_{k}, v_k,v_k'),\ cone(q_3,v_k,v_k',w_1,\ldots,w_{k-1}),\\
  & cone(q_3, w_{1},\ldots,w_{i}, v_i,\ldots,v_k) \hbox{ for } i\in \{1,\ldots, k-1\},\\
  & cone(q_3, w_{1},\ldots,w_{i},v_i',\ldots,v_k') \hbox{ for } i\in \{1,\ldots, k-1\},\\
  & cone(u_1,\ldots,u_i, w_i,\ldots,w_{k-1}, v_k,v_k') \hbox{ for } i\in \{1,\ldots, k-1\},\\
  & cone (u_1,\ldots,u_i, w_i,\ldots,w_{j}, v_{j},\ldots,v_k) \hbox{ for } i, j\in \{1,\ldots, k-1\},\ i \leq j,\\
  & cone (u_1,\ldots,u_i, w_i,\ldots,w_{j}, v_{j}',\ldots,v_k') \hbox{ for } i, j\in \{1,\ldots, k-1\},\ i \leq j.
\end{align*}
where all the rays are as defined in~\eqref{eq_important_vec_even}.
\end{proposition}

\begin{remark}
Similarly as in section~\ref{section_central_fibre_odd} for $n$ odd, one can analyse the structure of the central fibre of crepant resolutions and the flops between different resolutions for $n= 2k$. Here $(k+1)^2$ crepant resolutions can be pictured in the form of an isosceles triangle with $2k+1$ resolutions at the base and the number of resolution in consecutive rows parallel to the base decreasing by~2. Flops can be performed between adjacent resolutions in rows and in columns. The $G$-$\Hilb$ resolution can be identified based on the proof of~\cite[Thm~5.2]{NdCS} as the central resolution on the base of the triangle.
\end{remark}

\section{Reducible representations}\label{section_reducible}

Let $G\subset \SL(3,\bC)$ be a finite non-abelian subgroup such that $\bC^{3}$ decomposes into a 2-dimensional and 1-dimensional representation of~$G$. Equivalently, we may start from the 2-dimensional representation, that is $\ovl{G}\subset \GL(2,\bC)$ isomorphic to $G$, and construct $G$ as the set of $3\times 3$ matrices of the following form, for all $g$ in $\ovl{G}$:
\begin{equation*}
\begin{pmatrix}
g & 0 \\
0 & \det(g)^{-1}
\end{pmatrix} 
\end{equation*}

Let $X\to \bC^{3}/G$ be a projective crepant resolution. In this section we find the presentation of the Cox ring $\cR(X)$ using Theorem~\ref{theorem_valuation_lifting} and apply the result to describe the structure of the set of crepant resolutions for two examples.

We recall the classification result for finite subgroups of $\GL(2,\bC)$, see e.g.~\cite[Sect.~3]{CohenComplex}. We use the following notation. Let $G_1,G_2$ be matrix groups with normal subgroups $H_1, H_2$, and $\varphi$ be the isomorphism of $G_1 /H_1$ and $G_2/H_2$. Set $G_1 \times_{\varphi} G_2 = \{(g_1,g_2) \in G_1 \times G_2 \ | \ \varphi(g_1 H_1) = g_2 H_2\}$. Then by $(G_1 \mid H_1;\ G_2 \mid H_2)$ we understand
$\psi(G_1 \times_{\varphi} G_2)$, where $\psi$ is the multiplication map.  By~$\mu_{d}$ we denote the group of $d$-th roots of unity; we use the embedding $\mu_d \cdot I_n \subset \GL(n,\bC)$.

\begin{lemma}\label{lemma:classification-gl2c}
A finite group $\ovl{G} \subset \GL(2,\bC)$ is of the form $(\mu_{wd}\mid \mu_{d}; H\mid K)$ for some integers $w,d$ and a normal subgroup $K\subset H = (\bC^{*}\cdot \ovl{G})\cap \SL(2,\bC)$.
\end{lemma}

\begin{corollary}\label{corollary:classification-gl2c}
For a finite subgroup $\ovl{G} \subset \GL(2,\bC)$ we have the equality $[\ovl{G},\ovl{G}] = [H,H]$, hence also $\bC[x,y]^{[\ovl{G},\ovl{G}]} = \bC[x,y]^{[H,H]}$. Moreover, an element $p\in \bC[x,y]^{[\ovl{G},\ovl{G}]}$ homogeneous with respect to the degree grading is homogeneous with respect to the $\Ab(\ovl{G})$-action if and only if it is homogeneous with respect to the $\Ab(H)$-action.
\end{corollary}

A finite non-abelian group $H \subset \SL(2,\bC)$ is conjugate to one~of (see e.g.~\cite[Sect.~3]{CohenComplex}):

\begin{itemize}
\item binary dihedral group $\BD_{4q} =
\left\langle\left(\begin{smallmatrix}
0 & i\\ i & 0\end{smallmatrix}\right), \
\left(\begin{smallmatrix}\zeta_{n} & 0 \\ 0 & \zeta_{n}^{-1}\end{smallmatrix}\right)\right\rangle$ of order $4q,\ q\ge 2$,
\item binary tetrahedral group $\BT =
\left\langle \BD_8,\ \frac{1}{\sqrt{2}}
\left(\begin{smallmatrix}\zeta_{8} & \zeta_{8}^{3}\\ \zeta_{8} & \zeta_{8}^{7} \end{smallmatrix}\right) \right\rangle$ of order 24, 
\item binary octahedral group $\BO = 
\left\langle \BT,\ \left(\begin{smallmatrix}\zeta_{8}^3 & 0\\ 0 & \zeta_{8}^{5} \end{smallmatrix}\right) \right\rangle$ of order 48,

\item binary icosahedral group $\BI = \left \langle\frac{1}{\sqrt{5}}\left(\begin{smallmatrix}\zeta_{5}^{4}-\zeta_{5} & \zeta_{5}^{2} - \zeta_{5}^{3}\\ \zeta_{5}^2 - \zeta_{5}^3 & \zeta_{5}-\zeta_{5}^4\end{smallmatrix}\right), \ \frac{1}{\sqrt{5}}\left(\begin{smallmatrix}\zeta_{5}^{2}-\zeta_{5}^{4} & \zeta_{5}^{4} - 1\\ 1 - \zeta_{5} & \zeta_{5}^3-\zeta_{5}\end{smallmatrix}\right) \right\rangle$ of order 120.
\end{itemize}

Consider the ring of invariants $\bC[x,y,z]^{[G,G]} =  \bC[x,y]^{[\ovl{G}, \ovl{G}]}[z]$. The generators for $\ovl{G} \subset \SL(2,\bC)$, given below, can be found e.g. in~\cite{MillerBlichfeldtDickson}. Their homogeneity can be verified directly.

\begin{lemma}\label{lemma:invariants-gl2c}
For a finite group $\ovl{G} \subset \GL(2,\bC)$ the ring of invariants $\bC[x,y]^{[\ovl{G}, \ovl{G}]}$ is generated by three elements $p_{1}, p_{2}, p_{3}$, homogeneous with respect to both the standard degree grading and the $\Ab(\ovl{G})$-action, satisfying a trinomial relation.

{\scriptsize
\begin{center}
{\renewcommand{\arraystretch}{1.5}
\begin{tabular}{c|c|c}
$H$ & generators $p_{1}, p_{2}, p_{3}$ & relation \\
\hline
$\BD_{4q}$ & $x^{q}+y^{q}, \ x^{q}-y^{q}, \ xy$ & $Z_{1}^{2} - Z_{2}^{2} - 4Z_{3}^{q}$\\
\hline
$\BT$  & $x^{4} + y^4 + \sqrt{-12}x^{2}y^{2},\ x^{4} + y^4 -\sqrt{-12}x^{2}y^{2},\ x^{5}y-xy^{5}$ & $Z_{1}^3 - Z_{2}^3 - 12(\zeta_{3}-\zeta_{3}^{2})Z_{3}^{2}$ \\
\hline
$\BO$ & $x^{5}y-xy^{5},\ x^{8} + 14x^4y^4 + y^8,\ x^{12} - 33x^8y^4 - 33x^4y^8 + y^{12}$  & $108Z_{1}^4 - Z_{2}^3 + Z_{3}^{2}$ \\
\hline
$\BI$ & $\begin{array}{c}x^{11}y + 11x^6y^6 - xy^{11},\\
x^{20}-228x^{15}y^5+494x^{10}y^{10}+228x^5y^{15}+y^{20},\\
x^{30}+522x^{25}y^5-10005x^{20}y^{10}-10005x^{10}y^{20}-522x^5y^{25}+y^{30}
\end{array}$
& $1728Z_{1}^5+Z_{2}^3-Z_{3}^2$
\end{tabular}
}
\end{center}
}
\end{lemma}

Consider the ring homomorphism $\kappa\colon \bC[Z_1,Z_2,Z_3,Z_{4}]\to \bC[x,y,z]^{[G,G]}$ defined by $Z_{j}\mapsto p_{j}$ for $j=1,2,3$, and $Z_{4}\mapsto z$. Note that we have an $\Ab(H)$-action on $\bC[Z_{1},Z_{2},Z_{3},Z_{4}]$ induced from $\bC[x,y,z]$, where $H = (\bC^{*}\cdot \ovl{G})\cap \SL(2,\bC)$ acts on $x,y$ linearly and trivially on~$z$.

Let $g_{1},\ldots, g_{m}$ be representatives of all junior conjugacy classes of~$G$, $r_1,\ldots,r_m$ their orders and $\nu_{1},\ldots, \nu_{m}$ corresponding monomial valuations on~$\bC(x,y,z)$. We will also use monomial valuations $\wt{\nu}_{1},\ldots, \wt{\nu}_{m}$ on $\bC[Z_{1},Z_{2},Z_{3},Z_{4}]$ defined by setting $\wt{\nu}_{i}(Z_j) = \nu_{i}(\kappa(Z_{j}))$. 

We can now state the main result of this section. 

\begin{theorem}\label{theorem:cox-rings-reducible}
The Cox ring $\cR(X)$ of a crepant resolution $X \ra \bC^3/G$ is generated by $m+4$ elements
$$p_{j}\prod_{i=1}^{m}t_{i}^{\nu_{i}(p_{j})} \hbox{ for } j = 1,2,3,\qquad z\prod_{i=1}^{m}t_{i}^{\nu_{i}(z)}, \qquad t_{i}^{-r_{i}} \hbox{ for }  i=1,\ldots, m.$$
\end{theorem}

The proof follows from Theorem~\ref{theorem_valuation_lifting} by showing the lifting condition for valuations $\nu_i$ and $\wt{\nu}_i$. We begin with introducing some more notation and proving a lemma, which actually implies that a single valuation~$\nu_i$ can be lifted from $\bC(x,y,z)$ to $\bC(Z_1,Z_2,Z_3,Z_{4})$ via~$\kappa$ (in the sense of Theorem~\ref{theorem_valuation_lifting}). Then we finish with the main part of the argument, which is to show that valuations can be lifted simultaneously.

If we choose a monomial valuation $\nu$ associated with a matrix~$g$ which acts diagonally in coordinates corresponding to ring variables then we may write any polynomial  $F$ as $F_{0}+F_{1}$, where all monomials in $F_{0}$ have valuation equal to $\nu(F)$ and $\nu(F_{1}) > \nu(F)$. 

Denote by~$R$ the trinomial relation between $p_1, p_2, p_3$, that is the generator of $\ker \kappa$. Choose a valuation $\nu := \nu_i$, and if necessary, make a linear change of coordinates such that the corresponding matrix $g_i$ acts diagonally. In these coordinates we may take minimal parts of generators with respect to $\nu_i$: $p_{1,0}, p_{2,0}, p_{3,0}$. We also decompose $R$ with respect to $\wt{\nu} := \wt{\nu}_i$: $R=R_0+R_1$. Note that $R_0$ is a relation between $p_{1,0}, p_{2,0}, p_{3,0}$. In this setting we state the following result.

\begin{lemma}\label{lemma:relations-of-initials}
Every relation between $p_{1,0},p_{2,0},p_{3,0}$ homogeneous with respect to $\Ab(H)$-action lies in the ideal generated by $R_{0}$ in $\bC[Z_{1},Z_{2},Z_{3}]$.
\end{lemma}

\begin{proof}
First notice that $p_{1,0},p_{2,0},p_{3,0}$ depend only on the
inequality between values of $\nu$ on coordinates $u,v$ diagonalizing~$g_i$, that is whether $\nu(u) > \nu(v)$, $\nu(u) = \nu(v)$ or $\nu(u)<\nu(v)$. If $g_i = \zeta_dh$ for some $h \in H$ then the inequality between $\nu(u)$ and $\nu(v)$ is the same as the inequality between values on $u,v$ either of the valuation corresponding to $h$ or the one corresponding to $h^{-1}$. Thus without loss of generality we may assume that $p_{1,0}, p_{2,0}, p_{3,0}$ are determined by monomial valuation $\nu_h$ corresponding to $h \in H$.

Now we reduce to the case of Cox rings of minimal resolutions of $\bC^2/H$ for $H \subset \SL(2,\bC)$. By~\cite[Thm~6.12]{CoxSurf} we find that $p_1, p_2, p_3$ satisfy valuation lifting property. Here we reverse the implication from Theorem~\ref{theorem_valuation_lifting}: if $\phi_1,\ldots,\phi_s$ can be used for producing a generating set of the Cox ring then they satisfy the valuation lifting condition. This is because for any $\Ab(H)$-homogeneous $f \in \bC[x_1,\ldots,x_n]^{[H,H]}$ a presentation of $\ovl{\Theta}(f)$ in terms of $\ovl{\Theta}(\phi_1),\ldots,\ovl{\Theta}(\phi_s)$ produces a correct lift $\wt{f}$.

Hence it suffices to show that for $\nu_h$ the valuation lifting property implies the statement of the lemma. Let $Q \in \bC[Z_{1},Z_{2},Z_{3}]$ be any $\Ab(H)$-homogeneous relation between $p_{1,0},p_{2,0},p_{3,0}$. We can shift it by adding an element $W(Z_{1},Z_{2},Z_{3})R$, to obtain a correct lift $\wt{q} = Q+WR$ of $q = Q(p_1,p_2,p_3)$. That is, $\wt{\nu}_h(Q+WR) \geq \nu_h(q)$. But $\wt{\nu}_h(Q) < \nu_h(q)$ because $Q$ is a relation between leading forms of $p_1,p_2,p_3$ with respect to $\nu_h$. Thus the leading forms of $Q$ and $WR$ with respect to $\wt{\nu}_h$ must cancel: $-Q_0 = (WR)_0 = W_0R_0$. Since $Q_0$ is a relation between $p_{1,0},p_{2,0},p_{3,0}$, $Q_1$ also is, and if $Q_1 \neq 0$ we proceed by induction, repeating the argument for $Q_1$, to show that $Q$ is a multiple of $R_0$.

\end{proof}

\begin{proof}[Proof of theorem~\ref{theorem:cox-rings-reducible}]
We show that the valuation lifting condition from Theorem~\ref{theorem_valuation_lifting} is satisfied. Note that by definition of a monomial valuation it suffices to check only elements of $\bC[x,y,z]^{[G,G]}$ which are homogeneous with respect to the standard degree grading and with respect to $\Ab(G)$-action, i.e., by Corollary~\ref{corollary:classification-gl2c}, precisely the elements homogeneous with respect to standard grading and $\Ab(H)$-action. 

Let $f\in \bC[x,y,z]^{[G,G]}$ be any element homogeneous with respect to both the standard degree and the $\Ab(H)$-action. Take any $F\in \kappa^{-1}(f) \subset \bC(Z_1,Z_2,Z_3,Z_{4})$. Set $N = \max_{i=1,\ldots,m}(\nu_{i}(f)-\wt{\nu}_{i}(F))$. We may assume $N > 0$, since $N \ge 0$ by definition of $\wt{\nu}_{i}$, and the valuation lifting condition is equivalent to $N = 0$. We proceed by induction: we find $F'$ such that
\begin{enumerate}[label=(\arabic*)]
\item $\kappa(F') = f$,
\item $\wt{\nu}_{i}(F')\ge \wt{\nu}_{i}(F)$ for all $i$,
\item $\wt{\nu}_{i_{0}}(F') > \wt{\nu}_{i_{0}}(F)$ for some $i_{0}$.
\end{enumerate}
Let $j$ be any index such that $\wt{\nu}_{j}(F) < \nu_{j}(f)$. We decompose $F = F_{0}+F_{1}$ with respect to~$\wt{\nu}_{j}$. Changing coordinates if necessary, we may assume that the corresponding matrix $g_j$ acts diagonally and decompose $p_{i} = p_{i,0} + p_{i,1}$ with respect to~$\nu_{j}$. 

Since $\wt{\nu}_j(F) < \nu_{j}(f)$, the part of $F$ with smallest $j$-th valuation must annihilate parts of generators with smallest $j$-th valuation in order to increase valuation when passing through~$\kappa$. Formally, we have $F_{0}(p_{1,0},p_{2,0},p_{3,0},z) = 0$, but since $z$ is algebraically independent of $p_{1,0},p_{2,0},p_{3,0}$, in fact $F_{0}\in \bC[Z_{1},Z_{2},Z_{3}]$ is a relation between $p_{1,0},p_{2,0},p_{3,0}$. Moreover, as a sum of monomials in an element of $\bC[Z_{1},Z_{2},Z_{3},Z_{4}]$ homogeneous with respect to $\Ab(H)$-action, $F_{0}$ is homogeneous with respect to $\Ab(H)$-action.

By Lemma~\ref{lemma:relations-of-initials} applied to $\nu_j$ and $\wt{\nu}_j$ we have $F_{0} = PR_{0}$ for some $P\in \bC[Z_{1},Z_{2},Z_{3}]$. We claim that $F' = F-PR$ satisfies conditions (1)-(3) with $i_{0}= j$. Condition (1) is immediate since $\kappa(R) = 0$. Condition~(3) holds because $F' = F - PR = F_1-R_1P$ and $\wt{\nu}_{j}(R_1) > \wt{\nu}_{j}(R_{0})$. To prove condition~(2) we use again $F' = F_1 - R_1P$ and the fact that $R$ is a trinomial. We have to show that $\wt{\nu}_{i}(R_1) \ge \wt{\nu}_{i}(R_{0})$. Since $R_0$ consists of at least two monomials (as a relation between leading forms), $R_{1}$ is either 0 or a monomial. Assuming the latter and repeating the same argument for $\wt{\nu}_{i}$ for $i \neq j$ we see that there are at least two monomials in $R$ with valuation $\wt{\nu}_{i}(R)$. Thus at least one of them is in $R_{0}$, which implies $\wt{\nu}_{i}(R_{1}) \ge \wt{\nu}_{i}(R) = \wt{\nu}_i(R_{0})$.
\end{proof}

\begin{remark}
The argument generalizes immediately to an analogous description of Cox ring of minimal model of quotient singularity $\bC^{n+2}/G$ for finite non-abelian subgroup $G\subset \SL(n+2,\bC)$, acting on $\bC^{n+2}$ via a representation which splits into one 2-dimensional component and $n$ components of dimension~1.
\end{remark}

\begin{remark}
Some ideas in the above proof are related to the algorithm presented in~\cite[Sect.~4]{Yamagishi}. In the notation therein, $F_{0}=\min_{j}(F),\ p_{i,0} = \min_{j}(p_{i})$ and $\min_{j}(I) = (R_{0}) = \min_{j}(J)$. One may also check that all steps of the algorithm in~\cite[Sect.~4]{Yamagishi} end without introducing additional generators which gives a different proof of Theorem~\ref{theorem:cox-rings-reducible}.
\end{remark}

\begin{remark}
Note that to get the generator of the ideal of relations between elements generating the Cox ring $\cR(X)$ listed in Theorem~\ref{theorem:cox-rings-reducible} it suffices to take the trinomial generator of relations between $p_{1},p_{2},p_{3}$ and homogenise it with respect to $\Ab(G)$-action using variables mapped into generators of the form $t_{i}^{-r_{i}}$.
\end{remark}

The first application of Theorem~\ref{theorem:cox-rings-reducible} was the case of dihedral group, obtained in the odd case as $(\mu_{4}\mid \langle 1 \rangle;\ \BD_{4n}\mid \bZ_{n})$ and in the even case as $(\mu_{4}\mid \langle 1 \rangle;\ \BD_{2n}\mid \bZ_{n/2})$, studied in section~\ref{section_dihedral}. Here we provide two examples with an interesting feature, not appearing in the series of dihedral groups: the groups contain elements of age~2.

\begin{example}
Let $G = (\mu_{4}\mid \mu_{2};\ \BD_{16}\mid \BD_8) = \left \langle \left(\begin{smallmatrix}0 & i\\ i & 0 \end{smallmatrix}\right), \left(\begin{smallmatrix}\zeta_{8}^{3}& 0\\ 0 & \zeta_{8}\end{smallmatrix}\right) \right \rangle $. Using GAP~\cite{GAP4} we compute that~$G$ has order~16 and~six nontrivial conjugacy classes, five of which have age~1.
One may choose the following representatives of junior classes:
\begin{equation*}
g_1 = \left(\begin{matrix}0 & -1  & 0\\ 1 & 0 & 0\\ 0 & 0 & 1\end{matrix}\right),\qquad g_2 =\left(\begin{matrix}0 & \zeta_8 & 0\\ \zeta_8^7 & 0 & 0\\ 0 & 0 & -1 \end{matrix}\right)
\end{equation*}
and  $g_{3} = \diag(-1,-1,1), \ g_{4} = \diag(-i,i,1), \ g_{5} = \diag(\zeta_{8}^3,\zeta_8, -1)$.

Here $H = \BD_{16}$, so $[G,G] \simeq [\ovl{G},\ovl{G}] = [\BD_{16}, \BD_{16}] = \left\langle\left(\begin{smallmatrix}i & 0 \\ 0 & -i\end{smallmatrix}\right)\right\rangle$. By Lemma~\ref{lemma:invariants-gl2c}, the generators of algebra of invariants $\bC[x,y]^{[\ovl{G},\ovl{G}]}$ homogeneous with respect to $\Ab(H)$-action are $p_{1}=x^4+y^4, \ p_2= x^4-y^4,\ p_3 = xy$ with relation $Z_1^2-Z_2^2-4Z_3^4$. Diagonalizing representatives of conjugacy classes of age~1 we compute values of corresponding monomial valuations $\nu_{1},\ldots,\nu_5$ on generators $p_1,p_2,p_3$ and $z$ of $\bC[x,y,z]^{[G,G]}$:
\begin{center}
\begin{tabular}{c|cccc}
val\textbackslash gen & $x^4+y^4$ & $x^4-y^4$ & $xy$ & $z$ \\
\hline
$\nu_{1}$ & 4 & 6 & 2 & 0 \\
$\nu_{2}$ & 1 & 0 & 0 & 1\\
$\nu_3$ & 4 & 4 & 2 & 0 \\
$\nu_4$ & 4 & 4 & 4 & 0 \\
$\nu_5$ & 4 & 4 & 4 & 4
\end{tabular}
\end{center}
By Theorem~\ref{theorem:cox-rings-reducible}, the Cox ring of a crepant resolution $X\to \bC^{3}/G$ is isomorphic to $\bC[Z_{1},Z_{2},Z_{3},Z_{4},T_{1},\ldots,T_{5}]/(Z_{1}^2 T_2- Z_2^2 T_1 - 4Z_3^4T_4^2T_5)$ with degree matrix
\begin{equation*}
\left( \begin{array}{cc|cc|ccc|cc}
4 & 0  & 6 & -4 & 2 & 0  & 0  & 0 & 0\\
1 & -2 & 0 & 0  & 0 & 0  & 0  & 1 & 0\\
4 & 0  & 4 & 0  & 2 & 0  & 0  & 0 & -2\\
4 & 0  & 4 & 0  & 4 & -4 & 0  & 0 & 0\\
4 & 0  & 4 & 0  & 4 & 0  & -8 & 4 & 0
\end{array} \right),
\end{equation*}
where the first three groups of columns correspond to variables in monomials of the relation, and the last two columns to variables~$Z_4$ and~$T_3$, not involved in the relation.

As described in section~\ref{section_methods_Mov}, we compute the movable cone of~$X$ and its chamber decomposition. The rays of $\Mov(X) \subset \bR^5$ are
\begin{equation*}
\begin{array}{ccc}
v_{1} = (0,0,0,0,1), & v_{2} = (0,1,0,0,4), & v_{3} = (1,0,1,1,1),\\
v_{4} = (1,0,1,2,2), & v_{5} = (2,1,2,4,4), & v_{6} = (3,0,2,2,2),\\
v_{7} = (4,1,3,4,4), & v_{8} = (4,1,4,4,4), & v_{9} = (12,3,8,8,12).
\end{array}
\end{equation*} 

Let $w_1 = (1,0,1,1,2), \ w_2 = (3,0,2,2,6), \ w_3 = (6,1,4,4,8)$. Then the~11 (simplicial) chambers in $\Mov(X)$, corresponding to all crepant resolutions of $\bC^3/G$,~are:
\begin{equation*}
\begin{array}{cc}
\sigma_1 =cone(v_6,v_7,v_8,v_9,w_3), & \sigma_2 =cone(v_2,v_7,v_8,v_9,w_3), \\ \sigma_3 =cone(v_4,v_6,v_7,v_8,w_3), & \sigma_4 =cone(v_2,v_4,v_7,v_8,w_3), \\ \sigma_5 =cone(v_4,v_6,v_8,w_1,w_3), & \sigma_6 =cone(v_4,v_6,w_1,w_2,w_3),\\ \sigma_7 =cone(v_2,v_4,v_8,w_1,w_3), & \sigma_8 =cone(v_2,v_4,v_5,v_7,v_8), \\ \sigma_9 =cone(v_2,v_4,w_1,w_2,w_3),& \sigma_{10} =cone(v_3,v_4,v_6,v_8,w_1), \\ \sigma_{11} =cone(v_1,v_2,v_4,w_1,w_2).&
\end{array}
\end{equation*}

Flops between resolutions, i.e. pairs of adjacent chambers, are shown in the diagram below; a label $v/w$ at an edge means that in the set of rays $v$ is replaced by $w$.

\begin{equation*}
\xymatrix{
& & \sigma_{10} \ar@{-}[d]^{v_3 / w_3}&\\ 
\sigma_1 \ar@{-}[r]^{v_9 / v_4} \ar@{-}[d]^{v_6 / v_2} & \sigma_3\ar@{-}[d]^{v_6 / v_2} \ar@{-}[r]^{v_7 / w_1} & \sigma_5 \ar@{-}[r]^{v_8 / w_2} \ar@{-}[d]^{v_6 / v_2}& \sigma_6 \ar@{-}[d]^{v_6 / v_2} \\
\sigma_2 \ar@{-}[r]^{v_9 / v_4} & \sigma_4 \ar@{-}[d]^{w_3 / v_5}\ar@{-}[r]^{v_7 / w_1} & \sigma_7 \ar@{-}[r]^{v_8 / w_2} & \sigma_9 \ar@{-}[r]^{w_3 / v_1} & \sigma_{11}\\
& \sigma_8 
}
\end{equation*}

\end{example}

\begin{example}
Consider $G = (\mu_{8} \mid \mu_{4};\ \BD_{12} \mid C_{6})= \left\langle \left(\begin{smallmatrix}\zeta_{6} & 0\\ 0 & \zeta_6^5\end{smallmatrix} \right), \ \left(\begin{smallmatrix}0 & \zeta_{8}^{3}\\ \zeta_8^3 & 0\end{smallmatrix} \right)\right\rangle$ of order~24. It has seven conjugacy classes $[g_{1}],\ldots, [g_7]$ of age $1$ and four of age~$2$.

The representatives of junior classes are $g_1 = \diag(-1,-1,1), \ g_2 = \diag(\zeta_6^5,\zeta_6,1), \ g_3 = \diag(\zeta_3^2,\zeta_3,1), \ g_4 = \diag(i,i,-1), \ g_5 = \diag(\zeta_{12},\zeta_{12}^5,-1)$ and
\begin{equation*}
g_6 = \left(\begin{matrix}0 &\zeta_8^5 & 0\\ \zeta_8^5 & 0 & 0\\ 0 & 0 & i\end{matrix}\right),\qquad  g_7 = \left(\begin{matrix}0 &\zeta_8 & 0\\ \zeta_8 & 0 & 0\\ 0 & 0 & i\end{matrix}\right).
\end{equation*}

In this case $H = \BD_{12}$, so $[G,G] \simeq [\ovl{G},{\ovl{G}}] =  [\BD_{12}, \BD_{12}] = \left\langle\left(\begin{smallmatrix}\zeta_{3} & 0 \\ 0 & \zeta_{3}^{2}\end{smallmatrix}\right)\right\rangle$. Lemma~\ref{lemma:invariants-gl2c} implies that generators of the ring of invariants $\bC[x,y]^{[\ovl{G},\ovl{G}]}$ homogeneous with respect to $\Ab(H)$-action are $p_1=x^3+y^3, \ p_2= x^3-y^3,\ p_3 = xy$ with relation $Z_1^2-Z_2^2-4Z_3^3$. As before, diagonalizing representatives of conjugacy classes of elements of age~1 we compute values of corresponding monomial valuations $\nu_{1},\ldots,\nu_7$ on generators of $\bC[x,y,z]^{[G,G]}$.
\begin{center}
\begin{tabular}{c|cccc}
val\textbackslash gen & $x^3+y^3$ & $x^3-y^3$ & $xy$ & $z$ \\
\hline
$\nu_{1}$ & 3 & 3 & 2 & 0 \\
$\nu_{2}$ & 3 & 3 & 6 & 0\\
$\nu_3$ & 3 & 3 & 3 & 0\\
$\nu_4$ & 3 & 3 & 2 & 2\\
$\nu_5$ & 3 & 3 & 6 & 6\\
$\nu_6$ & 7 & 3 & 2 & 2\\
$\nu_7$ & 3 & 7 & 2 & 2\\
\end{tabular}
\end{center}
By Theorem~\ref{theorem:cox-rings-reducible}, the Cox ring of a crepant resolution $X\to \bC^{3}/G$ is isomorphic to $\bC[Z_{1},Z_{2},Z_{3},Z_{4},T_{1},\ldots,T_{7}]/(Z_{1}^2 T_6- Z_2^2 T_7 - 4Z_3^3T_2^2 T_3 T_5)$ with degree matrix
\begin{equation*}
\left( \begin{array}{cc|cc|cccc|ccc}
3 &  0 & 3 &  0 & 2 &  0 &  0 &  0  & 0 & -2 & 0 \\
3 &  0 & 3 &  0 & 6 & -6 &  0 &  0  & 0 &  0 & 0 \\
3 &  0 & 3 &  0 & 3 &  0 & -3 &  0  & 0 &  0 & 0 \\
3 &  0 & 3 &  0 & 2 &  0 &  0 &  0  & 2 &  0 & -4\\
3 &  0 & 3 &  0 & 6 &  0 &  0 & -12 & 6 &  0 & 0 \\ 
7 & -8 & 3 &  0 & 2 &  0 &  0 &  0  & 2 &  0 & 0 \\
3 &  0 & 7 & -8 & 2 &  0 &  0 &  0  & 2 &  0 & 0
\end{array} \right),
\end{equation*}
where the first three groups of columns correspond to variables in monomials of the relation and the last three columns to variables $Z_4, T_1,T_4$, not involved in the relation.

The movable cone $\Mov(X)\subset \bR^7$ has~17 rays:
\begin{equation*}
\begin{array}{cccc}
(0,0,0,1,1,1,1), & (0,0,0,1,3,1,1), & (0,0,0,3,3,3,7), & (0,0,0,3,3,7,3), \\
(1,1,1,1,1,1,1), & (1,1,1,1,3,1,1), & (2,3,3,2,3,2,2), & (2,3,3,2,6,2,2), \\
(2,3,3,3,3,3,3), & (2,3,3,3,3,3,7), & (2,3,3,3,3,7,3), & (2,6,3,2,6,2,2), \\
(2,6,3,6,6,6,6), & (2,6,3,6,6,6,14), & (2,6,3,6,6,14,6), & (3,3,3,3,3,3,7), \\
(3,3,3,3,3,7,3). & & & 
\end{array}
\end{equation*} 

It is subdivided into~34 chambers, corresponding to all crepant resolutions of~$\bC^3/G$.

\end{example}

\section{Examples of irreducible representations}\label{section_irreducible}

While for reducible 3-dimensional representations we have proved that the Cox ring of a crepant resolution is defined by a single trinomial equation, the irreducible case is much more interesting from the point of view of the structure of the Cox ring. The representations considered in section~\ref{section_reducible} decompose as a sum of a 2-dimensional and a 1-dimensional component. Thus the quotient space $\bC^3/G$ inherits a natural 2-dimensional torus action (where each $\bC^*$ factor comes from a homothety on a component), which lifts to a resolution. This means that we obtain a (non-compact) $T$-variety of complexity one. The spectrum of the Cox ring for such varieties is always defined by trinomial relations; in the projective case it is shown in~\cite{compl1}, and the case of non-compact rational varieties is treated in~\cite{HaWro}. Since the structure of $[G,G]$-invariants for reducible representations is quite simple (there are always just four invariants), in section~\ref{section_reducible} we always obtain a single trinomial relation in the Cox ring. But in general the structure of the Cox ring of a (crepant) resolution of a 3-dimensional quotient singularity can be more complex, see the case of $A_4$ in~\cite[Thm~4.5]{quotsingcox}.

Thus we intended to investigate the irreducible case not only to understand the geometry of the resolutions, but also to get new insight into the structure of the Cox ring in this much more intriguing setting. To be able to finish our computations, we have chosen the smallest possible groups with interesting properties. We restricted to groups with elements of age~2 in order to work with singularities where not much is known about the set of crepant resolutions. We also tried to find groups with the number of conjugacy classes of age~1 as small as possible in order to minimize the number of monomial valuations involved in the process of constructing generators of the Cox ring. We used a simple script and the library of small groups in GAP~\cite{GAP4} to test all finite groups up to order~256 and check which have irreducible faithful 3-dimensional representations with desired properties. We completed computations of the Cox ring for three groups and the results are presented below.

The first considered matrix group has order 21. It is a trihedral group: a group generated by diagonal matrices and a permutation matrix of a cycle of length~3. Moreover, it is the smallest example of a 3-dimensional representation with elements of age~2; see also~\cite[Ex.~5.5]{superpotentials}. The second group, also trihedral, has order 27. It is a representation of the Heisenberg group, consisting of upper-triangular matrices over $\bZ_3$ with 1's on the diagonal. The third one has order 54 and it is a double extension of the Heisenberg group.

For all three groups we compute the ring of invariants of the commutator subgroup $[G,G]$ and find its generating set consisting of eigenvectors of the $\Ab(G)$ action. The main difficulty is in the next step: extending this generating set to a (minimal) one satisfying the valuation lifting property from Theorem~\ref{theorem_valuation_lifting}. We use the following simple algorithm:
\begin{enumerate}
\item for a fixed (standard) degree, starting from the smallest,  compute the space of linear relations between leading forms $L_{1,i},\ldots,L_{k,i}$ of $[G,G]$-invariants $P_1,\ldots,P_k$ in this degree for different monomial valuations $\nu_i$ corresponding to junior classes,
\item intersect spaces of relations computed for different valuations in order to check if there is a relation $R$ which increases values of more than one valuation simultaneously (i.e. $R(L_{1,i},\ldots,L_{k,i}) = 0 = R(L_{1,j},\ldots,L_{k,j})$ for some valuations $\nu_i, \nu_j$),
\item check whether such a relation $R$ produces an element $R(P_1,\ldots,P_k)$ not satisfying valuation lifting property with respect to current generating set -- if yes then add it to the set of generators,
\item check whether the current generating set is minimal with the valuation lifting property (adding a new element may cause some redundancies),
\item check whether the spectrum of the ring determined by current generating set has smooth GIT quotients -- if not then go back to step 1, increasing the degree.
\end{enumerate}

Finally, one has to verify that the ring generated by the obtained set is really a Cox ring. This can be done either with the Singular~\cite{Singular} library \texttt{quotsingcox.lib} \cite{quotsingcox} accompanying~\cite{CompCox} or with the algorithm for finding the Cox ring of minimal models of quotient singularities from~\cite{Yamagishi}. In each of presented cases we try both methods and at least one works (that is, computations end in a reasonable amount of time). Note that, however, the algorithm from~\cite{Yamagishi} does not behave very well in the case when the candidate for the generating set of the Cox ring is not correct, hence we use it only for verification, not for determining elements of the generating set.

In the following section $G$ always denotes currently investigated matrix group and~$X$ is a crepant resolution of $\bC^3/G$.

\subsection{The 21-element group}\label{section_21}
Consider a trihedral group~$G$ generated~by
\begin{equation*}
    \left(\begin{array}{ccc}
    \zeta_7 & 0 & 0 \\
    0 & \zeta_7^2 & 0 \\
    0 & 0 & \zeta_7^4
    \end{array}\right),  \qquad
    \left(\begin{array}{ccc}
    0 & 0 & 1 \\
    1 & 0 & 0 \\
    0 & 1 & 0
    \end{array}\right),
\end{equation*}
where $\zeta_7$ is the seventh root of unity. The commutator subgroup is $[G,G] \simeq \bZ_7$, generated by the first group generator.

Determining a generating set of $\cR(X)$ for this group seems feasible at the first glance, because there are just 4 nontrivial conjugacy classes: 3 of age~1 and 1 of age~2, the last one containing the cube of the first group generator. Hence we have to find generators corresponding to $[G,G]$-invariants and add just 3 other ones, corresponding to exceptional divisors. However, $\bC[x,y,z]^{[G,G]}$ needs already at least 13 generators, in degrees from~3 to~7. The requirement of being eigenvectors with respect to the action of $\Ab(G) \simeq \bZ_3$ causes that they cannot be taken monomials.

Our computations show that the initial set of $[G,G]$-invariants which are $\Ab(G)$-eigenvectors, returned by the library~\cite{quotsingcox} based on Singular's computation of finite group invariants, is almost suitable for constructing the generating set of $\cR(X)$. It suffices to modify one invariant in degree~6 and three in degree~7 by a correction term, which is a product of lower degree generators, to increase their valuations (associated with conjugacy classes of elements of order~3). Thus, we have~16 generators of $\cR(X)$ in total.

\begin{proposition}\label{gen_21}
  The following set of generators of $\bC[x,y,z]^{[G,G]}$ satisfies valuation lifting property, i.e. it produces a generating set of the Cox ring $\cR(X)$ via Theorem~\ref{theorem_valuation_lifting}.
  \begin{align*}
& F_1 = xyz, \qquad G_1 = xy^3+x^3z+yz^3,\\
& G_2 = (-\zeta_3-2)xy^3+(2\zeta_3+1)x^3z+(-\zeta_3+1)yz^3,\\
& G_3 = (\zeta_3-1)xy^3+(-2\zeta_3-1)x^3z+(\zeta_3+2)yz^3,\\
& H_1 = x^3y^2+y^3z^2+x^2z^3, \qquad H_2 = \zeta_3x^3y^2+(-\zeta_3-1)y^3z^2+x^2z^3,\\
& H_3 = (-\zeta_3-1)x^3y^2+\zeta_3y^3z^2+x^2z^3, \qquad L_1 = x^5y+y^5z+xz^5 - 3x^2y^2z^2,\\
& L_2 = \zeta_3x^5y+(-\zeta_3-1)y^5z+xz^5, \qquad L_3 = (-\zeta_3-1)x^5y+\zeta_3y^5z+xz^5,\\
& M_1 = x^7+y^7+z^7 - x^2y^4z - x^4yz^2 - xy^2z^4,\\
& M_2 = (-3\zeta_3-2)x^7+(\zeta_3+3)y^7-7x^2y^4z+(7\zeta_3+7)x^4yz^2-7\zeta_3xy^2z^4+(2\zeta_3-1)z^7,\\
& M_3 = (-3\zeta_3-1)x^7+(\zeta_3-2)y^7+7x^2y^4z+7\zeta_3x^4yz^2+(-7\zeta_3-7)xy^2z^4+(2\zeta_3+3)z^7.
  \end{align*}
  The matrix of values of the monomial valuations is (columns corresponds to generators, ordered as above)
$$\left(\begin{array}{ccccccccccccc}
7 & 7 & 7 & 7 & 7 & 7 & 7 & 7 & 7 & 7 & 7 & 7 & 7\\
0 & 0 & 2 & 1 & 0 & 2 & 1 & 3 & 2 & 1 & 3 & 2 & 4\\
0 & 0 & 1 & 2 & 0 & 1 & 2 & 3 & 1 & 2 & 3 & 4 & 2
\end{array}\right)$$
\end{proposition}

\begin{proof}
  One preforms necessary computations along the scheme explained above. The linear algebra part, that is investigating relations between leading forms, is done using simple scripts in Macaulay2~\cite{M2} and the smoothness check is done using the Singular library~\cite{quotsingcox}. However, in this case checking that the constructed ring is a Cox ring, is hard. The algorithm in the library~\cite{quotsingcox} is not efficient enough to finish the computation on a standard computer (the problematic step is computing preimages of generators of the ideal of relations, which is bigger and more complex that in the next two cases). One can use the algorithm in~\cite{Yamagishi}, though some parts of it have to be implemented in Macaulay2 and some in Singular due to certain restrictions of these systems.
\end{proof}

We also determine the subdivision of $\Mov(X)$ into chambers, using methods and tools described in section~\ref{section_methods_Mov}.

\begin{proposition}\label{proposition_21_trihedral}
The number of chambers in the GIT chamber subdivision of the cone $\Mov(X)$, i.e. the number of projective crepant resolutions of $\bC^3/G$, is~4. There is a central chamber, from which one can pass to each of the remaining (corner) ones, which are not connected to each other.
\end{proposition}

\subsection{The 27-element group}
Considered trihedral group $G$ is generated by
\begin{equation*}
    \left(\begin{array}{ccc}
    1 & 0 & 0 \\
    0 & \zeta_3 & 0 \\
    0 & 0 & \zeta_3^2
    \end{array}\right),  \qquad
    \left(\begin{array}{ccc}
    0 & 0 & 1 \\
    1 & 0 & 0 \\
    0 & 1 & 0
    \end{array}\right),
\end{equation*}
where $\zeta_3$ is the third root of unity. The commutator subgroup $[G,G] \simeq \bZ_3$ is generated just by $\zeta_3I_3$.

There are 10 nontrivial conjugacy classes: 9 of age~1 and 1 of age~2, the last one containing $\zeta_3^2I_3$. Hence we have 9 monomial valuations to consider. The computations are simplified slightly by the fact that some of them share the same coordinate set -- there just~4 different coordinate sets.  Also, the valuation corresponding to $\zeta_3I_3$ is just the same as the lowest degree of a monomial in a polynomial, which actually allows us to skip it in our considerations.

The ring $\bC[x,y,z]^{[G,G]}$ is easy to describe, since the commutator subgroup has such a simple form: it is generated by all 10 monomials in degree~3. One can also easily determine their combinations which are eigenvectors of $\Ab(G)$. In particular, a monomial $xyz$ is $G$-invariant, so it belongs to this set.

It turns out again that not much is needed to turn the set of eigenvectors of $\Ab(G)$ into a set generating the Cox ring~$\cR(X)$. One just needs to add three new generators, which are counterparts of $xyz$ in other coordinate sets diagonalizing elements of~$G$. Then one of the generators from the initial set, $x^3+y^3+z^3$, becomes unnecessary, that is we can generate it from the remaining ones with the valuation lifting property satisfied. Thus we are left with~12 generators.

\begin{proposition}\label{gen_27}
  The following set of generators of $\bC[x,y,z]^{[G,G]}$ satisfies valuation lifting property, i.e. it produces a generating set of the Cox ring $\cR(X)$.
  \begin{align*}
& F_1 = xyz, \qquad F_2 = x^3+y^3+z^3 - 3xyz,\\
&  F_3 = \zeta_3(x^3+y^3+z^3) - 3xyz, \qquad  F_4 = \zeta_3^2(x^3+y^3+z^3) - 3xyz,\\
& G_1 = \zeta_3^2x^3+\zeta_3y^3+z^3, \qquad G_2 = \zeta_3x^3+\zeta_3^2y^3+z^3,\\
& G_3 = xy^2+x^2z+yz^2, \quad G_4 = x^2y+y^2z+xz^2,\\
& G_5 = \zeta_3xy^2+\zeta_3^2x^2z+yz^2, \qquad G_6 = \zeta_3x^2y+\zeta_3^2y^2z+xz^2,\\
& G_7 = \zeta_3^2xy^2+\zeta_3x^2z+yz^2, \qquad G_8 = \zeta_3^2x^2y+\zeta_3y^2z+xz^2.
  \end{align*}
  The matrix of values of the monomial valuations is (columns corresponds to generators, ordered as above)
$$\left(\begin{array}{cccccccccccc}
3 & 0 & 0 & 0 &  0 & 0 &  1 & 2 &  1 & 2 &  1 & 2\\ 
3 & 0 & 0 & 0 &  0 & 0 &  2 & 1 &  2 & 1 &  2 & 1\\
0 & 3 & 0 & 0 &  2 & 1 &  1 & 2 &  0 & 0 &  2 & 1\\
0 & 3 & 0 & 0 &  1 & 2 &  2 & 1 &  0 & 0 &  1 & 2\\
0 & 0 & 3 & 0 &  2 & 1 &  2 & 1 &  1 & 2 &  0 & 0\\
0 & 0 & 3 & 0 &  1 & 2 &  1 & 2 &  2 & 1 &  0 & 0\\
0 & 0 & 0 & 3 &  1 & 2 &  0 & 0 &  1 & 2 &  2 & 1\\
0 & 0 & 0 & 3 &  2 & 1 &  0 & 0 &  2 & 1 &  1 & 2\\
3 & 3 & 3 & 3 &  3 & 3 &  3 & 3 &  3 & 3 &  3 & 3\\
\end{array}\right)$$
\end{proposition}

\begin{proof}
The linear algebra part and the smoothness check is done as in Proposition~\ref{gen_21}. In this case it can be checked that the constructed ring is a Cox ring using the geometric criterion and other procedures from the library developed for~\cite{CompCox}, or it follows by the positive result of the algorithm in~\cite{Yamagishi}. We perform both tests for this generating set.
\end{proof}

We compute the subdivision of $\Mov(X)$ into chambers as described in section~\ref{section_methods_Mov}.

\begin{proposition}
The number of chambers in the GIT chamber subdivision of the cone $\Mov(X)$, i.e. the number of projective crepant resolutions of $\bC^3/G$, is~5272.
\end{proposition}

\subsection{The 54-element group}

Now $G$ is a double extension of the group from the previous section. We need to add one more generator -- we take:
\begin{equation*}
      \left(\begin{array}{ccc}
    1 & 0 & 0 \\
    0 & \zeta_3 & 0 \\
    0 & 0 & \zeta_3^2
    \end{array}\right),  \qquad
    \left(\begin{array}{ccc}
    0 & 0 & 1 \\
    1 & 0 & 0 \\
    0 & 1 & 0
    \end{array}\right), \qquad
    \left(\begin{array}{ccc}
    -1 & 0 & 0 \\
    0 & 0 & -1 \\
    0 & -1 & 0 
    \end{array}\right).
\end{equation*}
The Heisenberg group is precisely the commutator subgroup. Note that~$G$ is not a trihedral group.

There are 7 conjugacy classes of age~1 and 2 of age~2 in~$G$. The ring of invariants of $[G,G]$ is simpler than in the previous case -- it has only~4 generators. However, this time we need more additional generators, hence finally we arrive at~9 necessary generators of $[G,G]$-invariants in degrees 3, 6, and 9 (where 4 generators in degree~3 are the same as in the previous case). The proof of the next proposition goes along the same lines as for Proposition~\ref{gen_27}.

\begin{proposition}\label{gen_54}
  The following set of generators of $\bC[x,y,z]^{[G,G]}$ satisfies valuation lifting property, i.e. it produces a generating set of the Cox ring $\cR(X)$.
  \begin{align*}
& F_1 = xyz, \qquad F_2 = x^3+y^3+z^3 - 3xyz,\\
&  F_3 = \zeta_3(x^3+y^3+z^3) - 3xyz, \qquad  F_4 = \zeta_3^2(x^3+y^3+z^3) - 3xyz,\\
& G_1 = 3(xyz)^2 + \zeta_3xyz(x^3+y^3+z^3) + \zeta_3^2(x^3y^3+x^3z^3+y^3z^3),\\
& G_2 = 3(xyz)^2 + \zeta_3^2xyz(x^3+y^3+z^3) + \zeta_3(x^3y^3+x^3z^3+y^3z^3),\\
& G_3 = 3(xyz)^2 + xyz(x^3+y^3+z^3) + x^3y^3+x^3z^3+y^3z^3,\\
& G_4 = (xyz)^2 - 3(x^3y^3+x^3z^3+y^3z^3),\\
& H_1 = x^6y^3-x^3y^6-x^6z^3+y^6z^3+x^3z^6-y^3z^6.
  \end{align*}
  The matrix of values of the monomial valuations is (columns corresponds to generators, ordered as above)
$$\left(\begin{array}{ccccccccc}
  1 & 1 & 1 & 1 & 0 & 0 & 0 & 0 &  0\\
3 & 3 & 3 & 3 & 6 & 6 & 6 & 6 &  12\\
3 & 0 & 0 & 0 & 3 & 3 & 3 & 0 &  3\\
0 & 3 & 0 & 0 & 3 & 3 & 0 & 3 &  3\\
0 & 0 & 3 & 0 & 3 & 0 & 3 & 3 &  3\\
0 & 0 & 0 & 3 & 0 & 3 & 3 & 3 &  3\\
3 & 3 & 3 & 3 & 6 & 6 & 6 & 6 &  9
  \end{array}\right)$$
\end{proposition}

\begin{proposition}
The number of chambers in the GIT chamber subdivision of the cone $\Mov(X)$, i.e. the number of projective crepant resolutions of $\bC^3/G$, is~755.
\end{proposition}

\section*{Appendix}

Here we collect some data on the structure of central fibres of crepant resolutions of $\bC^3/G$ for the 21-element group $G$ from section~\ref{section_21}. By Proposition~\ref{proposition_21_trihedral} the $\Mov$ cone decomposes into the central chamber $\sigma_0$ and three \emph{corner} chambers $\sigma_1, \sigma_2, \sigma_3$. By~$I$ we denote the ideal of $\Spec \cR\subset \bC^{16}$ given by generators produced from the invariants in Proposition~\ref{gen_21}, as in Theorem~\ref{theorem_valuation_lifting}.

We start with a description of the components of the subset $S_0 \subset \Spec \cR(X)$ mapped to the central fibre (its ideal can be computed as in section~\ref{section_central_fibre_odd}). Then we present four tables of orbits of the action of $(\bC^{*})^{16}$ on $\bC^{16}$, which cover $S_0$ and are stable with respect to a linearisation chosen from a chamber interior. In each table, \emph{equations} are vanishings of coordinates describing the orbit, $dim$ is the dimension of the orbit, and $dim(\cap)$ is the dimension of the intersection of the orbit with $\Spec \cR(X)$.

\begin{proposition}
The subset $S_0 \subset \Spec \cR(X)\subset \bC^{16} = \Spec \bC[T_1,\ldots,T_{16}]$ mapped to the central fibre of a resolution $X\to \bC^{3}/G$ has three components
\begin{eqnarray*}
& V(I+T_{14}), \ V(T_{1},T_{2},T_{4},T_{5},T_{6},T_{7},T_{9},T_{10},T_{11},T_{12},T_{15}),\\
& V(T_1,T_2,T_3,T_5,T_6,T_7,T_9,T_{10},T_{11},T_{13},T_{16}).
\end{eqnarray*}
and two more, which are unstable for any chamber.
General points on the first one are always stable; it corresponds to an exceptional divisor. General points on the second one are stable only for~$\sigma_{1}$ and general points on the third one are stable only for~$\sigma_{2}$.
\end{proposition}
\begin{proof}
One decomposes, e.g. in Singular~\cite{Singular}, the ideal generated by~$I$ and the generators of the ring of invariants of the Picard torus action (which can be computed in 4ti2~\cite{4ti2}).
\end{proof}

\bigskip
\subsection*{Stable orbits for the central fibre for {$\sigma_0$}}
\leavevmode

{\small
\begin{tabular}{c|c|c}
equation & dim & dim ($\cap$)\\
\hline
 \hline $T_{14} = 0$ & $15$ & $5$ \\ 
 \hline $T_{1} = T_{14} = 0$ & $14$ & $4$ \\ 
 \hline $T_{2} = T_{14} = 0$ & $14$ & $4$ \\ 
 \hline $T_{3} = T_{14} = 0$ & $14$ & $4$ \\ 
 \hline $T_{4} = T_{14} = 0$ & $14$ & $4$ \\ 
 \hline $T_{1} = T_{5} = T_{6} = T_{7} = T_{8} = T_{9} = T_{10} = T_{14} = 0$ & $8$ & $4$ \\ 
 \hline $T_{1} = T_{2} = T_{3} = T_{4} = T_{5} = T_{6} = T_{7} = T_{8} = T_{9} = T_{10} = T_{14} = 0$ & $5$ & $3$ \\ 
 \hline $T_{2} = T_{5} = T_{7} = T_{9} = T_{10} = T_{11} = T_{12} = T_{14} = T_{15} = 0$ & $7$ & $4$ \\ 
 \hline $T_{2} = T_{3} = T_{5} = T_{7} = T_{9} = T_{10} = T_{11} = T_{12} = T_{14} = T_{15} = 0$ & $6$ & $3$ \\ 
 \hline $T_{2} = T_{5} = T_{6} = T_{9} = T_{10} = T_{11} = T_{13} = T_{14} = T_{16} = 0$ & $7$ & $4$ \\ 
 \hline $T_{2} = T_{4} = T_{5} = T_{6} = T_{9} = T_{10} = T_{11} = T_{13} = T_{14} = T_{16} = 0$ & $6$ & $3$ \\ 
 \hline $T_{1} = T_{2} = T_{5} = T_{6} = T_{7} = T_{9} = T_{10} = T_{14} = T_{15} = T_{16} = 0$ & $6$ & $4$ \\ 
 \hline $T_{1} = T_{2} = T_{5} = T_{6} = T_{7} = T_{8} = T_{9} = T_{10} = T_{14} = T_{15} = T_{16} = 0$ & $5$ & $3$ \\ 
 \hline $T_{1} = T_{2} = T_{5} = T_{6} = T_{7} = T_{9} = T_{10} =$ & & \\
 $=T_{11} = T_{12} = T_{13} = T_{14} = T_{15} = T_{16} = 0$ & $3$ & $3$ \\
\end{tabular}
}

\bigskip
\subsection*{Stable orbits for the central fibre for  {$\sigma_1$}}
\leavevmode

{\small
\begin{tabular}{c | c | c}
equation & dim & dim ($\cap$)\\
\hline
 \hline $T_{14} =  0$ & $15$ & $5$ \\ 
 \hline $T_{1} = T_{14} =  0$ & $14$ & $4$ \\ 
 \hline $T_{2} = T_{14} =  0$ & $14$ & $4$ \\ 
 \hline $T_{3} = T_{14} =  0$ & $14$ & $4$ \\ 
 \hline $T_{4} = T_{14} =  0$ & $14$ & $4$ \\ 
 \hline $T_{1} = T_{5} = T_{6} = T_{7} = T_{8} = T_{9} = T_{10} = T_{14} =  0$ & $8$ & $4$ \\ 
 \hline $T_{1} = T_{2} = T_{3} = T_{4} = T_{5} = T_{6} = T_{7} = T_{8} = T_{9} = T_{10} = T_{14} =  0$ & $5$ & $3$ \\ 
  \hline $T_{1} = T_{2} = T_{4} = T_{5} = T_{6} = T_{7} = T_{9} = T_{10} = T_{11} = T_{12} = T_{15} =  0$ & $5$ & $4$ \\ 
 \hline $T_{1} = T_{2} = T_{4} = T_{5} = T_{6} = T_{7} = T_{9} = T_{10} = T_{11} = T_{12} = T_{13} = T_{15} =  0$ & $4$ & $3$ \\ 
 \hline $T_{2} = T_{5} = T_{7} = T_{9} = T_{10} = T_{11} = T_{12} = T_{14} = T_{15} =  0$ & $7$ & $4$ \\ 
 \hline $T_{2} = T_{3} = T_{5} = T_{7} = T_{9} = T_{10} = T_{11} = T_{12} = T_{14} = T_{15} =  0$ & $6$ & $3$ \\ 
 \hline $T_{1} = T_{2} = T_{5} = T_{6} = T_{7} = T_{9} = T_{10} = T_{14} = T_{15} = T_{16} = 0$ & $6$ & $4$ \\ 
 \hline $T_{1} = T_{2} = T_{5} = T_{6} = T_{7} = T_{8} = T_{9} = T_{10} = T_{14} = T_{15} = T_{16} =  0$ & $5$ & $3$ \\ 
 \hline $T_{1} = T_{2} = T_{4} = T_{5} = T_{6} = T_{7} = T_{9} = $ & & \\
 $=T_{10} = T_{11} = T_{12} = T_{14} = T_{15} = T_{16} =  0$ & $3$ & $3$
\end{tabular}
}

\vfill
\newpage

\subsection*{Stable orbits for the central fibre for  {$\sigma_2$}}
\leavevmode

{\small
\begin{tabular}{c|c|c}
equation & dim & dim ($\cap$)\\
\hline
 \hline $T_{14} = 0$ & $15$ & $5$ \\ 
 \hline $T_{1} = T_{14} = 0$ & $14$ & $4$ \\ 
 \hline $T_{2} = T_{14} = 0$ & $14$ & $4$ \\ 
 \hline $T_{3} = T_{14} = 0$ & $14$ & $4$ \\ 
 \hline $T_{4} = T_{14} = 0$ & $14$ & $4$ \\ 
 \hline $T_{1} = T_{5} = T_{6} = T_{7} = T_{8} = T_{9} = T_{10} = T_{14} = 0$ & $8$ & $4$ \\ 
 \hline $T_{1} = T_{2} = T_{3} = T_{4} = T_{5} = T_{6} = T_{7} = T_{8} = T_{9} = T_{10} = T_{14} = 0$ & $5$ & $3$ \\ 
 \hline $T_{1} = T_{2} = T_{3} = T_{5} = T_{6} = T_{7} = T_{9} = T_{10} = T_{11} = T_{13} = T_{16} = 0$ & $5$ & $4$ \\ 
 \hline $T_{1} = T_{2} = T_{3} = T_{5} = T_{6} = T_{7} = T_{9} = T_{10} = T_{11} = T_{12} = T_{13} = T_{16} = 0$ & $4$ & $3$ \\ 
 \hline $T_{2} = T_{5} = T_{6} = T_{9} = T_{10} = T_{11} = T_{13} = T_{14} = T_{16} = 0$ & $7$ & $4$ \\ 
 \hline $T_{2} = T_{4} = T_{5} = T_{6} = T_{9} = T_{10} = T_{11} = T_{13} = T_{14} = T_{16} = 0$ & $6$ & $3$ \\ 
 \hline $T_{1} = T_{2} = T_{5} = T_{6} = T_{7} = T_{9} = T_{10} = T_{14} = T_{15} = T_{16} = 0$ & $6$ & $4$ \\ 
 \hline $T_{1} = T_{2} = T_{5} = T_{6} = T_{7} = T_{8} = T_{9} = T_{10} = T_{14} = T_{15} = T_{16} = 0$ & $5$ & $3$ \\ 
 \hline $T_{1} = T_{2} = T_{3} = T_{5} = T_{6} = T_{7} = T_{9} = $ & & \\
 $= T_{10} = T_{11} = T_{13} = T_{14} = T_{15} = T_{16} = 0$ & $3$ & $3$ \\
\end{tabular}
}

\bigskip
\subsection*{Stable orbits for the central fibre for  {$\sigma_3$}}
\leavevmode

{\small
\begin{tabular}{c|c|c}
equation & dim & dim ($\cap$)\\
\hline
 \hline $T_{14} = 0$ & $15$ & $5$ \\ 
 \hline $T_{1} = T_{14} = 0$ & $14$ & $4$ \\ 
 \hline $T_{2} = T_{14} = 0$ & $14$ & $4$ \\ 
 \hline $T_{3} = T_{14} = 0$ & $14$ & $4$ \\ 
 \hline $T_{4} = T_{14} = 0$ & $14$ & $4$ \\ 
 \hline $T_{1} = T_{5} = T_{6} = T_{7} = T_{8} = T_{9} = T_{10} = T_{14} = 0$ & $8$ & $4$ \\ 
 \hline $T_{1} = T_{2} = T_{3} = T_{4} = T_{5} = T_{6} = T_{7} = T_{8} = T_{9} = T_{10} = T_{14} = 0$ & $5$ & $3$ \\ 
 \hline $T_{5} = T_{6} = T_{7} = T_{8} = T_{9} = T_{10} = T_{11} = T_{12} = T_{13} = T_{14} = 0$ & $6$ & $4$ \\ 
 \hline $T_{1} = T_{5} = T_{6} = T_{7} = T_{8} = T_{9} = T_{10} = T_{11} = T_{12} = T_{13} = T_{14} = 0$ & $5$ & $3$ \\ 
 \hline $T_{2} = T_{5} = T_{7} = T_{9} = T_{10} = T_{11} = T_{12} = T_{14} = T_{15} = 0$ & $7$ & $4$ \\ 
 \hline $T_{2} = T_{3} = T_{5} = T_{7} = T_{9} = T_{10} = T_{11} = T_{12} = T_{14} = T_{15} = 0$ & $6$ & $3$ \\ 
 \hline $T_{2} = T_{5} = T_{6} = T_{9} = T_{10} = T_{11} = T_{13} = T_{14} = T_{16} = 0$ & $7$ & $4$ \\ 
 \hline $T_{2} = T_{4} = T_{5} = T_{6} = T_{9} = T_{10} = T_{11} = T_{13} = T_{14} = T_{16} = 0$ & $6$ & $3$ \\ 
 \hline $T_{2} = T_{5} = T_{6} = T_{7} = T_{8} = T_{9} = T_{10} = $ & & \\ 
 $ = T_{11} = T_{12} = T_{13} = T_{14} = T_{15} = T_{16} = 0$ & $3$ & $3$ \\
\end{tabular}
}

\bibliographystyle{plain}
\bibliography{threedim}

\end{document}